\documentclass[11pt, oneside]{article}
\usepackage{amsfonts}
\usepackage{mathrsfs}
\usepackage{color}
\usepackage[colorlinks]{hyperref}
\usepackage{latexsym}
\usepackage{amssymb}
\usepackage{amsmath}
\usepackage{enumerate}
\usepackage{amsthm}
\usepackage{indentfirst}
\usepackage{mathtools}
\usepackage{tikz}
\usepackage{psfrag}
\usepackage{graphicx}
\usepackage{bm}
\usepackage[rflt]{floatflt}
\usepackage{float}
\usepackage{authblk}
\usepackage{ulem}
\usepackage[square, comma, sort&compress, numbers]{natbib}
\usepackage{verbatim}
\numberwithin{equation}{section}
\allowdisplaybreaks

\newenvironment{proof2.1}{\medskip\noindent{\bf Proof of the Theorem 2.1:}\enspace}{\hfill \qed \newline \medskip}

\newenvironment{proof2.2}{\medskip\noindent{\bf Proof of the Theorem 2.2:}\enspace}{\hfill \qed \newline \medskip}

\newtheorem{theorem}{\color{black}\indent Theorem}[section]
\newtheorem{lemma}{\color{black}\indent Lemma}[section]
\newtheorem{proposition}{\color{black}\indent Proposition}[section]
\newtheorem{definition}{\color{black}\indent Definition}[section]
\newtheorem{remark}{\color{black}\indent Remark}[section]

\usepackage{amssymb,amsmath}
\begin{document}
\title{\LARGE\bf {Classification of blow-up and global existence of solutions to an initial $\textrm{Neumann}$ problem
}
\thanks{
 The research was
 supported by NSFC (11301211).}}
\author[1]{Bin Guo\thanks{Corresponding author:
Email addresses: bguo@jlu.edu.cn(B. Guo)}}
\author[1]{Jingjing Zhang}
\author[1,2]{Menglan Liao}
\affil[1]{School of Mathematics, Jilin University, Changchun, Jilin Province 130012, China}
\affil[2]{Department of Mathematics, Michigan State University, East Lansing,

MI 48824, USA}
\renewcommand*{\Affilfont}{\small\it}
\date{} \maketitle
\vspace{-20pt}

{\bf Abstract:}\ The aim of this paper is to apply the modified potential well method and some new differential inequalities to
  study the asymptotic behavior of solutions to
the initial homogeneous $\hbox{Neumann}$ problem of a nonlinear
diffusion equation driven by the $p(x)$-\hbox{Laplace} operator.
Complete classification of global existence and blow-up in finite time of solutions is given when the initial data satisfies different conditions.
Roughly speaking, we obtain a threshold result for the solution to exist globally
or to blow up in finite time when the initial energy is subcritical and critical, respectively.
Further, the decay rate of the $L^2$ norm is also obtained for global solutions.
Sufficient conditions for the existence of global and blow-up solutions are also provided for supercritical initial energy. At last, we give two-sided estimates of asymptotic behavior when the diffusion term dominates the source.
This is a continuation of our previous work \cite{GG}.

{\bf MSC(2010):} Primary: 35K55, 35B40; Secondary: 35B44.

{\bf Keywords:} Complete classification; p(x)-Laplace operator; Asymptotic behavior.

\section {Introduction}
It is well known that
parabolic problems with  nonlocal or local terms can describe some
phenomena of real problems such as population dynamics, nuclear science and biological
sciences where the total mass is often conserved, but the growth of
a certain cell or the temperature is known to have some definite
form. For example,  $\hbox{Budd,
Hu~etc.}$ in \cite{CBBDAS,BHHMY} discussed  the blow-up
and global existence of the solution to the following initial and boundary value problem
\begin{equation}\label{equ101-01}
\begin{cases}
u_{t}=\hbox{div}(\varepsilon\nabla u)+|u|^{r-2}u-\frac{1}{|\Omega|}\int_{\Omega}|u|^{r-2}udx,&(x,t)\in Q_{T},\\
\frac{\partial u}{\partial \nu}(x,t)=0,&(x,t)\in\Gamma_{T},\\
u(x,0)=u_{0}(x),&x\in\Omega,
\end{cases}
\end{equation}
where $ Q_{T}=\Omega\times(0,T],~ \Omega\subset\mathbb{R}^N
(N\geq 1)$ is a bounded simply connected domain and
$0<T<\infty$,
$\Gamma_{T}$ denotes the lateral boundary of the cylinder $Q_{T}$, $\nu$ is the unit outward normal on $\partial\Omega$,
$\varepsilon$ is a diffusion coefficient. Later, in 2007, $\hbox{Soufi,~Jazar~and~Monneau}$ in
\cite{AESMJRM} obtained a global existence result and proved that the solution to Problem
\eqref{equ101-01} with a negative initial energy might blow up in finite time.
Further, Gao and Han in \cite{WJGYZH} also proved that the solution
to the problem above might
blow up in finite time for a suitable positive initial energy. However,  Model
(\ref{equ101-01}) may describe some characteristics of incompressible
fluid in a homogenous and isotropic medium, but it can not more
accurately describe the motions of some fluids such as filtration,
elastic mechanics and electromagnetic fluids in non-homogenous and
anisotropic medium. The main reason is that the characteristics of
such medium may vary in dependence on directions and points. Generally speaking, the coefficient
$\varepsilon$ is usually a function which may depend on the unknown
function $u$ or the gradient of $u$ rather than a fixed constant.
Especially, we consider the simple case when $\varepsilon=|\nabla
u|^{p(x)-2}$.  For more details about such problems with variable
exponents, the interested readers may refer to
\cite{MR,LDPHPM,XLF,DVCAF,SNASIS-1,SNASIS-2,YCSLMR,GAKM,LGYC} and the references therein. In this paper,
we consider the following evolution problem with variable exponents
\begin{equation}\label{equ101}
\begin{cases}
u_{t}=\mathrm{div}(|\nabla u|^{p(x)-2}\nabla u)+|u|^{r(x)-2}u-\frac{1}{|\Omega|}\int_{\Omega}|u|^{r(x)-2}udx,&(x,t)\in Q_{T},\\
\frac{\partial u}{\partial \nu}(x,t)=0,&(x,t)\in\Gamma_{T},\\
u(x,0)=u_{0}(x),&x\in\Omega.
\end{cases}
\end{equation}
Throughout the paper, we  assume that the exponents $p(x),~r(x)$ are continuous in $\overline{\Omega}$ with logarithmic module of continuity:
\begin{equation}\label{addequ101}
    \begin{split}
     |p(x)-p(y)|+ |r(x)-r(y)|\leq \omega(|x-y|),\quad \forall~x,~y\in\Omega,~|x-y|<1,
    \end{split}
\end{equation}
where $\omega(\tau)$ is a function with
\[\limsup\limits_{\tau\to 0^+}\omega(\tau)\ln\frac{1}{\tau}=C<+\infty.\]

In case of $p\equiv2$, Ferreira etc.\cite{RFAPMPJD} applied Kaplan's method to establish the non-global
existence and global existence results for Problem (\ref{equ101}). Later, $\hbox{Wu,~Guo and Gao}$ in \cite{XLWBGWJG}
constructed a control functional and applied suitable embedding
theorem to prove that the solution blows up in finite time for a
positive initial energy in the case when $r$ is a function with
respect to space variables. For $p\not\equiv2$, we need to emphasize
that it is not a trivial generalization of the similar problems in the
constant case. In dealing with such problems, we have to encounter some difficulties:

$\bullet$ Due to the gap between the norm and the modular, it is not straightforward to apply  potential well method used in
\cite{HAL} to construct two invariant sets. As we all know, in constant cases, a key question of applying potential well method is to analyze the properties of the function \begin{align*}
R(\lambda)=\lambda-B^{\beta}\lambda^{\beta},~\lambda>0,\beta>1,
\end{align*}
which helps us better explore the mechanism of how the source dominates the diffusion term or the diffusion term dominates the source.
 Therefore, for variable exponents case, the gap between the norm and the modular leads to that $R(\lambda)$ will be expressed as the following form
\begin{align*}
R(\lambda)=\lambda-\max\Big\{(B\lambda)^{\alpha},(B\lambda)^{\beta}\Big\},~\lambda>0,\alpha>\beta>1.
\end{align*}
Obviously, this function is only continuous but not differential for $\lambda>0$. So, in order to bypass this difficulty,
 one has to develop some new methods or put forward some new ideas.

$\bullet$ Another difficulty is  the lack of scaling technique and hence some methods
used in \cite{JXYCHJ,FCLCHX} are not applicable, which results in that one can not analyze the uniqueness and the asymptotic behavior of the solution.

$\bullet$ When we check the proofs step by step in the references, the following relations are used repeatedly
\begin{align*}
&\|\nabla u\|^{r}_{p,\Omega}\equiv\Big(\int_{\Omega}|\nabla
u|^{p}dx\Big)^{\frac{r}{p}};\\
&\hbox{div}(|\nabla(\lambda u)|^{p-2}\nabla(\lambda u))\equiv
\lambda^{p-1}\hbox{div}(|\nabla u|^{p-2}\nabla u).
\end{align*}
These are obviously untrue if $p$ is not a fixed constant.  It is easy to verify that some methods(upper and lower solution techniques) used in \cite{QBZ,BGWJG-3,WJLMXW} are hardly  available.
 In
this paper, we modify the classical potential well method and then combine $\hbox{Galerkin}$ method and Levine's argument with differential inequality techniques to give threshold results for the solutions to exist globally or to blow up in finite time
 when the initial energy functional is subcritical(critical). Subsequently,
 we can give an abstract criterion for the existence of global solutions that tend to $0$ as $t$ tends to $\infty$ or finite-time blow-up solutions for supercritical initial energy.
At the end of this paper, we analyze the properties of some differential inequalities and apply energy estimate method to obtain  two-sided estimates of asymptotic behavior of the solution.

The outline of this paper is: In Section 2, we introduce the
function spaces of $\rm{Orlicz}$ -Sobolev type, and some basic lemmas about the spaces. In Section 3, we give the definition
of the weak solution to the Problem \eqref{equ101} and some preliminary lemmas which are critical to prove our main results.
Section 4 is devoted to the discussion of properties of the solution
for $r(x)\geq p(x)$.  In Section 5, we study long-time asymptotic behavior of solutions for the case $r(x)< p(x)$.

\section {Basic spaces}
 In this section, we introduce some Banach spaces of $\rm{Orlicz-Sobolev}$ type.
 Set
\begin{equation*}
  C_+(\overline{\Omega})=\left\{h\in C(\overline{\Omega}):\min\limits_{x\in\overline{\Omega}}h(x)>1\right\}.
\end{equation*}
For any $h\in C_+(\overline{\Omega}),$ we define\\
\begin{equation*}
h^+=\sup\limits_{x\in\Omega}h(x)\quad\text{and}\quad h^-=\inf\limits_{x\in\Omega}h(x).
\end{equation*}
For any $p\in C_{+}(\overline{\Omega})$, we introduce the variable exponent Lebesgue space,
\begin{align*}
L^{p(x)}(\Omega):=\big\{ u:u~\text{is a measurable real-valued function},\int_{\Omega}|u(x)|^{p(x)}dx<\infty\big\},
\end{align*}
endowed with the Luxemburg norm
\begin{align*}
\|u\|_{p(.)}=\|u\|_{L^{p(x)}(\Omega)}=\inf\{\lambda>0:\int_{\Omega}\Big|\frac{u(x)}{\lambda}\Big|^{p(x)}dx\leq1\}.
\end{align*}
The dual space of $L^{p(x)}(\Omega)$ is $L^{p'(x)}(\Omega)$, where $\frac{1}{p(x)}+\frac{1}{p'(x)}=1$.
\[
\begin{split}
& V(\Omega)=\left\{u|\,u\in  W^{1,1}(\Omega),
\,|\nabla u|^{p(x)}\in
L^{1}(\Omega),\,\int_{\Omega}u(x)dx=0\;\right\},\\
&\|u\|_{V(\Omega)}=\|u\|_{2,~\Omega}
+\|\nabla u\|_{p(.),~\Omega},\\
& W(Q_T)=\left\{u:[0,T]\mapsto V(\Omega)|\,u\in
L^{2}(Q_{T}),~|\nabla u|^{p(x)}\in
L^{1}(0,T;L^{1}(\Omega))\right\},
\\
& \|u\|_{W(Q_T)}=\|\nabla u\|_{p(x),~Q_T},
\end{split}
\]
and denote by $W'(Q_T)$ the dual of $W(Q_T)$. That is,
\begin{align*}
\omega\in W'(Q_{T})\Longleftrightarrow\begin{cases}
\omega=\omega_{0}+\sum\limits_{k=1}^{n}\frac{\partial
\omega_{k}}{\partial x_{k}},~\omega_{0}\in
L^{2}(Q_{T}),~\omega_{k}\in
L^{q(x)}(Q_{T});\\
<\omega,\varphi>\overset{\Delta}=\iint_{Q_{T}}\Big(\omega_{0}\varphi+\sum\limits_{k=1}^{n}\omega_{k}\frac{\partial
\varphi}{\partial x_{k}}\Big)dxdt,~~\forall ~\varphi\in W(Q_{T}).
\end{cases}
\end{align*}
The norm in $W'(Q_{T})$ is equipped with
$\|\omega\|_{W'(Q_{T})}=\sup\{|<\omega,\varphi>|:\varphi\in
W(Q_{T}),~\|\varphi\|_{W(Q_{T})}\break\leq1\}.$
For the sake of simplicity, we first state some results about the
properties of the $\textrm{Luxemburg}$ norm.
\begin{lemma}[\cite{LDPHPM,XLF}]\label{lem101}
The space $L^{p(x)}(\Omega)$ is a separable, uniformly convex Banach~space and its conjugate space is $L^{p'(x)}(\Omega),$ where $\frac{1}{p(x)}+\frac{1}{p'(x)}=1$.

For any $u\in L^{p(x)}(\Omega)$ and $v\in L^{p'(x)}(\Omega),$ we have the following H\"{o}lder's inequality
\begin{equation*}
\begin{split}
\left|\int_{\Omega}uv dx\right|\leq \Big(\frac{1}{p^{-}}+\frac{1}{(p^{-})'}\Big)\|u\|_{p(.)}\|v\|_{p'(.)}\leq2\|u\|_{p(.)}\|v\|_{p'(.)}.
\end{split}
\end{equation*}
\end{lemma}
\begin{lemma}[\cite{LDPHPM,XLF}]\label{PRO102}
Define\\
\begin{equation*}
\begin{split}
\rho(u)=\int_{\Omega}|u|^{p(x)}dx~~~for ~all~ u\in L^{p(x)}(\Omega),
\end{split}\end{equation*}
then\\
\begin{equation*}
\begin{split}
(1)~&\rho(u)>1(~respectively,\rho(u)=1,\rho(u)<1) ~if ~and ~only ~if ~\|u\|_{p(.)}>1\\&(~respectively,~ \|u\|_{p(.)}=1, \|u\|_{p(.)}<1);\\
(2)~&if~\|u\|_{p(.)}>1,~then ~\|u\|_{p(.)}^{p^{-}}\leq\rho(u)\leq \|u\|_{p(.)}^{p^{+}};\\
(3)~&if~\|u\|_{p(.)}<1,~then ~\|u\|_{p(.)}^{p^{+}}\leq\rho(u)\leq\|u\|_{p(.)}^{p^{-}};\\
(4)~&if~ p_{1}(x)\leq p_{2}(x),~\text{then we have a continuous embedding }L^{p_{2}(x)}(\Omega)\hookrightarrow L^{p_{1}(x)}(\Omega).
\end{split}\end{equation*}
\end{lemma}
\begin{lemma}[\cite{LDPHPM,XLF}]\label{addlemma}
Let $p,~q\in C_+(\overline\Omega)$. Assume that $p^+<N$ and
\[q(x)\leq p^*(x):=\frac{Np(x)}{N-p(x)}\]
 there is a continuous embedding
$V(\Omega)\hookrightarrow L^{q(x)}(\Omega).$ Namely,
\begin{equation}\label{201}
  \|u\|_{q(x)}\leq B\|\nabla u\|_{p(x)},\quad\forall~u\in V(\Omega).
\end{equation}
where $B>0$ is the optimal constant of the embedding inequality above.
\end{lemma}
\section {Preliminaries}
 We first recall the definition of weak solutions to Problem $\eqref{equ101}$ and some results.
Because of the degeneracy, Problem $\eqref{equ101}$ does not admit
classical solutions in general. So, we introduce weak solutions in
the following sense.
\begin{definition}\label{def301}
A function $u(x,t)\in W(Q_T)\cap L^{\infty}(0,T;L^{2}(\Omega)),~u_{t}\in L^{2}(Q_{T})$
is called a weak solution of Problem $\eqref{equ101}$ if for every
test-function $$\xi\in \mathcal {Z}\equiv\{\eta(x,t):\eta\in
W(Q_T)\cap L^{\infty}(0,T;\,L^{2}(\Omega)),\eta_{t}\in
W'(Q_T)\},$$ and every $t_{1},t_{2}\in[0,T]$ the following identity
holds:
\begin{equation}\label{e201}
\begin{split}
&\int^{t_{2}}_{t_{1}}\int_{\Omega}\Big[u\xi_{t}-|\nabla
u|^{p(x)-2}\nabla u\nabla \xi+F(x,t,u)\xi\Big]dxdt
=\int_{\Omega}u\xi dx\Big|^{t_{2}}_{t_{1}},
\end{split}
\end{equation}
where
$F(x,t,u)=|u|^{r(x)-2}u-\frac{1}{|\Omega|}\int_{\Omega}|u|^{r(x)-2}udx.$
\end{definition}
\begin{definition} Let $u(x,t)$ be a weak solution of \eqref{equ101} with the initial data $u_{0}\in V(\Omega).$ Define  the maximal existence time by $T_{max}=T(u_{0})$.

(1)~If $u(x,t)\in V(\Omega)$ for $0\leqslant t<\infty$, then $T_{max}=+\infty;$

(2)~If $T_{max}<\infty$, then $\lim\limits_{t\rightarrow T_{max}}\|u\|_{V(\Omega)}=+\infty.$ In such case, we call that the solution blows up in finite time.

\end{definition}
In order to prove the existence of weak solutions, we need to find a
linear independent basis $\{\varphi_{k}(x)\}\subset V(\Omega).$
\begin{lemma}[\cite{BGWJG-4}, Lemma 2.3] Suppose that $p(x)$ satisfies the
constrains in (\ref{addequ101}), then
\begin{align*}
&(1)~V(\Omega)~is~ separable~ and ~reflexive;
\\
&(2)~C^{\infty}_{*}(\Omega)=\{u\in
C^{\infty}|\textstyle\int_{\Omega}u(x)dx=0\}~is~dense~in
~V(\Omega).
\end{align*}
\end{lemma}

Set $$ L^{2}_{*}(\Omega)=\{u(x)\in
L^{2}(\Omega)|\textstyle\int_{\Omega}u(x)dx=0\},~~V^+_{*}(\Omega)=\{u(x)|u\in
L_{*}^{2}(\Omega)\cap W^{1,1}(\Omega),~|\nabla u|\in L^{p^{+}}\}.$$
Since $V^+_{*}(\Omega)$ is separable, there exists a span of a
countable set of linearly independent functions
$\{\varphi_{k}(x)\}^{\infty}_{1}\subset V^+_{*}(\Omega).$
\begin{lemma}[\cite{BGWJG-4}, Lemma 2.4]  Suppose that $p(x)$ satisfies
$(\ref{addequ101})$. Then the set
$\{\varphi_{k}(x)\}^{\infty}_{1}$ is dense in $V(\Omega).$
\end{lemma}
\begin{lemma}[\cite{BGWJG-4}, Lemma 2.5] For $u\in W$ and
$\varepsilon>0$, there exist a sequence $\{c_{k}(t)\},~c_{k}(t)\in
C^{1}[0,T]$ and an integer $n$ such
that$$\|u-\sum\limits_{k=1}^{n}c_{k}(t)\varphi_{k}(x)\|_{W(Q_{T})}<\varepsilon.$$
\end{lemma}
\begin{remark}
{\rm In fact, without the condition $(\ref{addequ101})$, Inequality \eqref{201} and the density of $C^{\infty}_{*}(\Omega)$ may not hold. For more related results, the interested reader may refer to \cite[Chapter~8]{LDPHPM} and \cite[Chapters~2~and~4]{DVCAF}.}
\end{remark}
Next, we give some results from \cite{GG}.
\begin{theorem}[\cite{GG}, Theorem 3.1]\label{lem302}
Assume that $p(x),r(x)\in C_{+}(\bar{\Omega})$ satisfy \eqref{addequ101}, and the following conditions hold:
\begin{enumerate}[$(1)$]
  \item $E(0)<E_1,$\quad$\int_\Omega |\nabla u_0|^{p(x)}dx>\alpha_1;$
  \item $\max\Big\{1,\frac{2N}{N+2}\Big\}<p^-<N,$\quad $\max\{p^+,2\}<r^-\leq r(x)\leq r^+\leq \frac{2N+(N+2)p^-}{2N},$
\end{enumerate}
then the solution of Problem \eqref{equ101} blows up in finite time, here
\[\alpha_1=B_1^{\frac{r^+p^+}{p^+-r^-}},\quad B_1=B+1,\quad E_1=\Big(\frac{r^{-}-p^{+}}{p^{+}r^{-}}\Big)\alpha_1.\]
\end{theorem}
The following lemma plays a key role in our proofs.
\begin{lemma}\label{lem301}
If $u(x,t)\in W(Q_{T}),u_{t}\in L^{2}(Q_{T})$ is a weak solution to Problem $\eqref{equ101},$ then the energy functional $E(t)$ satisfies
\begin{equation}\label{3}
E(t)+\int_{0}^{t}\int_\Omega{u_s^2}dxds\leq E(0),~t\geq0,
\end{equation}
where
\begin{equation*}\label{2}
 E(t):=E(u(x,t))=\int_{\Omega}\frac{1}{p(x)}|\nabla u(x,t)|^{p(x)}dx-\int_{\Omega}\frac{1}{r(x)}|u(x,t)|^{r(x)}dx.
\end{equation*}
\end{lemma}
\begin{proof} For the convenience of the reader, we give a simple proof.
A weak solution $u(x,t)$ to Problem (\ref{equ101-01}) is a limit function
of the sequence of $\hbox{Galerkin's}$ approximation
$$u^{(m)}(x,t)=\sum\limits_{k=1}^{m}C_{k}^{(m)}(t)\varphi_{k}(x),~~\varphi_{k}(x)\in V(\Omega),~C_{k}^{(m)}(t)\in C^{1}[0,T),$$
where the coefficients $C_{k}^{m}(t)$ satisfy that
\begin{equation*}\int_{\Omega}[\frac{\partial u^{m}_{k}}{\partial t}\varphi_{k}(x)+|\nabla
u_{k}^{m}|^{p(x)-2}\nabla u_{k}^{m}\nabla\varphi_{k}(x)+F(x,t,u_{k}^{m})\varphi_{k}(x)]dx=0,~~\varphi_{k}(x)\in
V(\Omega).
\end{equation*}
 Noticing that $\varphi_{k}(x)\in
V(\Omega),$ it is not hard to verify for any fixed $m$
\begin{align}\label{addd103}
\int_{\Omega}u^{(m)}(x,t)dx=\int_{\Omega}\sum\limits_{k=1}^{m}C_{k}^{(m)}(t)\varphi_{k}(x)dx=\sum\limits_{k=1}^{m}C_{k}^{(m)}(t)\int_{\Omega}\varphi_{k}(x)dx=0.
\end{align}
Furthermore,
following the lines of the proof of Lemmas 2.6-2.7 in
\cite{BGWJG-4} as well as the help of Lemmas \ref{PRO102} and \ref{addlemma},   we know that there exists a positive constant
$C=C( |\Omega|,|u_{0}|_{V(\Omega)},p^{\pm},r,\break N)$ such that
\begin{align}\label{Ine302}
&~\|u^{(m)}\|_{W(Q_{T})}+\|u^{(m)}\|_{L^{\infty}(0,T;L^{2}(\Omega))}+\|u_{t}^{(m)}\|_{W'(Q_{T})}\leqslant
C.
\end{align}
From\cite{SNASIS-1}, we may get the following inclusions:
\begin{align*}
&u^{(m)}\in W(Q_T)\subseteq
L^{p^-}(0,T;W_{*}^{1,p^{-}}(\Omega)),\\
&u^{(m)}_{t}\in W'(Q_T)\subseteq
L^{q^{-}}(0,T; (V^{+}_{*}(\Omega))'),\\
&W_{*}^{1,p^-}(\Omega)\subset L_{*}^{r(x)}(\Omega)\subset
(V^{+}_{*}(\Omega))',
\end{align*}
where $W_{*}^{1,p^-}(\Omega)=\{u\in
W^{1,p^-}(\Omega),\textstyle\int_{\Omega}u(x)dx=0\}$.

From \cite[Corollary 6]{JSIMON}, it follows that the sequence
$\{u^{(m)}\}$ contains a subsequence strongly convergence in
$L^{r^{-}}(0,T;L^{r(x)}(\Omega))$ with some $1<r(x)<p^{-*}:Np^-/(N-p^-).$ This subsequence contains
a subsequence which converges to $u(x,t)$ a.e. in $Q_{T}$ (see,
 e.g.~\cite[Th 2.8.1]{AKJOSF}. These conclusions together with the uniform estimates in $m$ allow
one to extract from the sequence $u^{(m)}$ a subsequence (for the
sake of simplicity, we assume that it merely coincides with the
whole of the sequence) such that
\begin{equation}\label{addd301}
\begin{cases}
\begin{split}
&u^{(m)}\rightharpoonup u ~~weakly~in~~
W(Q_{T})~and~strongly~in~L^{r^{-}}(0,T;L^{r(x)}(\Omega));\\
&u^{(m)}\rightarrow u~~~~ a,e.~in~~~Q_{T};\\
&|\nabla u^{(m)}|^{p(x)-2}D_{i}u^{(m)}\rightharpoonup A_{i}(x,t)
~~~weakly~in~~~ L^{(p^{-})'}(0,T;L^{p'(x)}(\Omega)),
\end{split}
\end{cases}
\end{equation}
for some functions $u\in W(Q_{T}),~A_{i}(x,t)\in L^{(p^{-})'}(0,T;L^{p'(x)}(\Omega)).$
Similar as the proof of Theorem 2.1 in
\cite{BGWJG-4}, we have $A_{i}(x,t)=|\nabla u|^{p(x)-2}\frac{\partial u}{\partial x_{i}},~~a.e. (x,t)\in Q_{T}.$
In addition, (\ref{addd103}) and (\ref{addd301}) as well as $\textrm{Lebesgue}$ Dominated Convergence theorem  yield
\begin{align*}\int_{\Omega}u(x,t)dx=0,~~t>0.\end{align*}

Multiplying (\ref{e202}) by $(C_{k}^{(m)}(t))'$, summing over
$k=1,2,\cdots,m,$ we arrive at the relation
\begin{equation}\label{addd302}
\begin{split}
~\|u_{t}^{(m)}\|_{L^{2}(\Omega)}&+\frac{d}{dt}E(u^{m}(x,t))=0.
\end{split}\end{equation}
From (\ref{addd301}),(\ref{addd302})and Fatou's
lemma, it follows that
\begin{align*}
E(t_{2})+\int_{t_{1}}^{t_{2}}\int_{\Omega}|u_{s}|^{2}dxds\leq
E(t_{1}),~t_{2}\geqslant t_{1}\geq0.\end{align*}
\end{proof}
In this section, in addition to \eqref{addequ101}, we assume that the following condition holds:
\[(\mathbf{H})~\max\Big\{1,\frac{2N}{N+2}\Big\}<p^-<N;\quad \max\{p^+,2\}<r^-\leq r(x)\leq r^+\leq \frac{Np^-}{N-p^-}.\]

For $u\in  V(\Omega)$, define two functionals as follows:
\begin{equation*}\label{304}
J(u)=\int_{\Omega}\frac{1}{p(x)}|\nabla u|^{p(x)}dx-\int_{\Omega}\frac{1}{r(x)}|u|^{r(x)}dx;
\end{equation*}
\begin{equation*}
\begin{split}
I(u)=\int_{\Omega}|\nabla u|^{p(x)}dx-\int_{\Omega}|u|^{r(x)}dx,
\end{split}
\end{equation*}
and define the Nehari's manifold
\begin{equation*}\label{equ104}
\begin{split}
\mathcal{N}&=\{u\in  V(\Omega)\backslash\{0\}|I(u)=0\};\\
\mathcal{N}_{+}&=\{u\in  V(\Omega)|I(u)>0\}\cup\{0\};\\
\mathcal{N}_{-}&=\{u\in  V(\Omega)|I(u)<0\}.\\
\end{split}
\end{equation*}

Since $r^+\leq Np^-/(N-p^-)$, the functionals $J$ and $I$ are well defined and continuous on $u\in  V(\Omega)$. Define the depth of the potential well by
\[
d=\inf_{u\in \mathcal{N}}J(u).
\]

\begin{lemma}\label{LEM302}
The depth of the potential well $d$ is positive.
\end{lemma}
\begin{proof}
First, we prove the set $\mathcal{N}$ is not empty. Choose a nonzero function $\varphi(x)$ satisfying $\int_{\Omega}|\nabla\varphi|^{p(x)}dx>0,~\int_{\Omega}\varphi(x)dx=0.$
 For any $\lambda>0$, set $k(\lambda)=I(\lambda\varphi),$
 then it is not hard to verify that $k(\lambda)$ is continuous with respect to $\lambda$. Additionally, a simple computation shows that
\begin{equation*}
\begin{cases}
&k(\lambda)\geq \lambda^{p^{+}}\int_{\Omega}|\nabla\varphi|^{p(x)}dx-\lambda^{r^{-}}\int_{\Omega}|\varphi|^{r(x)}dx,~~\lambda<1;\\
&k(\lambda)\leq \lambda^{p^{+}}\int_{\Omega}|\nabla\varphi|^{p(x)}dx-\lambda^{r^{-}}\int_{\Omega}|\varphi|^{r(x)}dx,~~\lambda\geq 1.
\end{cases}
\end{equation*}
Again, by means of the condition $r^{-}>p^{+}$, one knows that there exist two positive constants $\lambda_0,~\lambda_1$ satisfying $0<\lambda_{0}\leq 1<\lambda_{1}<\infty$( which maybe depend on $\varphi$) such that
\begin{equation}\label{Add1001}
\begin{cases}&\lambda_{1}^{p^{+}}\int_{\Omega}|\nabla\varphi|^{p(x)}dx-\lambda_{1}^{r^{-}}\int_{\Omega}|\varphi|^{r(x)}dx<0,\\
&\lambda_{0}^{p^{+}}\int_{\Omega}|\nabla\varphi|^{p(x)}dx-\lambda_{0}^{r^{-}}\int_{\Omega}|\varphi|^{r(x)}dx>0,
\end{cases}\end{equation}
which imply that
\begin{align*}
k(\lambda_{0})>0,~k(\lambda_{1})<0.
\end{align*}
Moreover, by combining the above inequality with intermediate value theorem, we know that there exists a positive constant $\lambda^{*}>0$ ( which maybe depend on $\varphi$) such that $k(\lambda^{*})=0$ which shows that the set $\mathcal{N}$ is not empty.

 Next, we will complete the proof of this lemma.
Fix $u\in \mathcal{N}$. It follows from Lemma \ref{PRO102}, (\ref{201}) and the definition of $\mathcal{N}$ that
\begin{equation}\label{equ306}
\begin{split}
\int_{\Omega}|\nabla u|^{p(x)}dx&=\int_{\Omega}|u|^{r(x)}dx \leq \max\Big\{\|u\|_{r(x)}^{r^{+}} ,\|u\|_{r(x)}^{r^{-}}\Big\}\\
&\leq \max\Big\{B^{r^{+}}\|\nabla u\|_{p(x)}^{r^{+}},B^{r^{-}}\|\nabla u\|_{p(x)}^{r^{-}}\Big\}\\
& \leq \max\Bigg\{B^{r^{+}}\max\Big\{\big(\int_{\Omega}|\nabla u|^{p(x)}dx\big)^{\frac{r^{+}}{p^+}},\big(\int_{\Omega}|\nabla u|^{p(x)}dx\big)^{\frac{r^{+}}{p^-}}\Big\},\\
&~~~~B^{r^{-}}\max\Big\{\big(\int_{\Omega}|\nabla u|^{p(x)}dx\big)^{\frac{r^{-}}{p^+}},\big(\int_{\Omega}|\nabla u|^{p(x)}dx\big)^{\frac{r^{-}}{p^-}}\Big\}  \Bigg\}.
\end{split}\end{equation}
If~$\int_\Omega|\nabla u|^{p(x)}dx\geq 1$, then $\eqref{equ306}$~implies $
    \int_\Omega|\nabla u|^{p(x)}dx\geq\min\{B^{\frac{r^{-}p^-}{p^--r^-}},B^{\frac{r^{+}p^-}{p^--r^+}}\}.$\\
If~$\int_\Omega|\nabla u|^{p(x)}dx<1$, then $\eqref{equ306}$~implies
$\int_\Omega|\nabla u|^{p(x)}dx\geq \min\{B^{\frac{r^{-}p^+}{p^+-r^-}},B^{\frac{r^{+}p^+}{p^+-r^+}}\}.$

Therefore, we get
\begin{equation}\label{equ307}
 \int_\Omega |\nabla u|^{p(x)}dx\geq \min\Big\{B^{\frac{r^+p^-}{p^--r^+}},B^{\frac{r^-p^+}{p^+-r^-}}\Big\}>0.
\end{equation}

Noticing that $r^{-}>p^{+}$ and $u\in \mathcal{N}$, we have
\begin{equation*}
\begin{split}
J(u)&=\int_{\Omega}\frac{1}{p(x)}|\nabla u|^{p(x)}dx-\int_{\Omega}\frac{1}{r(x)}|u|^{r(x)}dx \geq \frac{1}{p^{+}}\int_{\Omega}|\nabla u|^{p(x)}dx-\frac{1}{r^{-}}\int_{\Omega}|u|^{r(x)}dx\\
&\geq\frac{r^{-}-p^{+}}{p^{+}r^{-}}\int_{\Omega}|\nabla u|^{p(x)}dx+\frac{1}{r^{-}}I(u)\geq \frac{r^{-}-p^{+}}{p^{+}r^{-}}\min\Big\{B^{\frac{r^+p^-}{p^--r^+}},B^{\frac{r^-p^+}{p^+-r^-}}\Big\},
\end{split}\end{equation*}
which indicates $d\geq \frac{r^{-}-p^{+}}{p^{+}r^{-}}\min\Big\{B^{\frac{r^+p^-}{p^--r^+}},B^{\frac{r^-p^+}{p^+-r^-}}\Big\}>0$ from the definition of $d$.
\end{proof}

For convenience, we introduce some notations.
For any $s>d$, define
\begin{align*}
J^{s}=\left\{u\in  V(\Omega)\Big|J(u)\leq s\right\}.
\end{align*}
By the definition of $J(u),~\mathcal{N},~J^{s}~\text{and}~d$, we see that
\begin{align*}
\mathcal{N}_{s}\triangleq \mathcal{N}\cap J^{s}=\left\{u\in N\Big|J(u)\leq s\right\}\neq \emptyset,\quad\forall~s>d.
\end{align*}
Also define
\begin{align}\label{4}
\lambda _{s}=\inf\{~\|u\|_{2}|u\in \mathcal{N}_{s}~\},\quad\Lambda_{s}=\sup\{~\|u\|_{2}|u\in \mathcal{N}_{s}\}.
\end{align}
It is clear that $\lambda_{s}$($\Lambda_{s}$) is nonincreasing(nondecreasing) with respect to $s$.
In addition, we introduce the following two sets:
\begin{align*}
\mathcal{U}&=\{u_{0}\in  V(\Omega)\Big|~\text{the solution}~u=u(t) ~\text{of}~\eqref{equ101}~\text{blows up in finite time}\};\\
\mathcal{G}&=\{u_{0}\in  V(\Omega)\Big|~\text{the solution}~u=u(t)~\text{of}~\eqref{equ101}~\text{tends to $0$ in}~ V(\Omega)~\text{as}~t\rightarrow \infty\}.
\end{align*}

To analyze the behavior of solutions to Problem \eqref{equ101}, we need the
 following properties of the functionals and sets defined above.
\begin{lemma}\label{lem304}
{\rm(1)~$0$ is away from both $\mathcal{N}$ and $\mathcal{N}_{-}$ , $i.e.$ $dist(0,\mathcal{N})>0,$ and $dist(0,\mathcal{N}_{-})>0$.

~~~~~~~~~~~~~~~~~~~(2) For any $s>d$, the set $J^{s}\cap \mathcal{N}_{+}$ is bounded in $V(\Omega)$.}
\end{lemma}

\begin{proof}
(1)~For any $u\in \mathcal{N}$, by the definition of $d$  we obtain
\begin{equation}\label{Key ine}
\begin{split}
d \leq J(u)&=\int_{\Omega}\frac{1}{p(x)}|\nabla u|^{p(x)}dx-\int_{\Omega}\frac{1}{r(x)}|u|^{r(x)}dx\\
&\leq \frac{1}{p^{-}}\int_{\Omega}|\nabla u|^{p(x)}dx-\frac{1}{r^{+}}\int_{\Omega}|u|^{r(x)}dx\\
& \leq \Big(\frac{1}{p^{-}} - \frac{1}{r^{+}}\Big)\int_{\Omega}|\nabla u|^{p(x)}dx+\frac{1}{r^{+}}I(u)\\
&= \Big(\frac{1}{p^{-}} - \frac{1}{r^{+}}\Big)\int_{\Omega}|\nabla u|^{p(x)}dx,
\end{split}\end{equation}
which implies that $$dist(0,\mathcal{N})=
\inf\limits_{u\in \mathcal{N}} \|u\|_{V(\Omega)}\geq\inf\limits_{u\in \mathcal{N}} \|\nabla u\|_{p(x),\Omega}\geq \min\Big\{\Big(\frac{p^{-}r^{+}}{r^{+}-p^{-}}d\Big)^{\frac{1}{p^+}},\Big(\frac{p^{-}r^{+}}{r^{+}-p^{-}}d\Big)^{\frac{1}{p^-}}\Big\}.$$

On the other hand, for any $u\in \mathcal{N}_{-}$, we have $I(u)<0$. Following the line of the proof of (\ref{equ307}), one may get
 \begin{align}\label{Key ine01}
 \int_{\Omega}|\nabla u|^{p(x)}dx\geq \min\Big\{B^{\frac{r^+p^-}{p^--r^+}},B^{\frac{r^-p^+}{p^+-r^-}}\Big\}>0,
 \end{align}
 which shows that
 $$dist(0,\mathcal{N}_{-})=\inf\limits_{u\in \mathcal{N}_{-}} \int_{\Omega}|\nabla u|^{p(x)}dx\geq \min\Big\{B^{\frac{r^+p^-}{p^--r^+}},B^{\frac{r^-p^+}{p^+-r^-}}\Big\}>0.$$
This completes the proof of the first part of this lemma.

(2)~For any $u\in J^{s}\cap \mathcal{N}_{+}$, we have
\begin{equation*}
\begin{split}
s\geq J(u)&=\int_{\Omega}\frac{1}{p(x)}|\nabla u|^{p(x)}dx-\int_{\Omega}\frac{1}{r(x)}|u|^{r(x)}dx\\
&\geq \frac{1}{p^{+}}\int_{\Omega}|\nabla u|^{p(x)}dx-\frac{1}{r^{-}}\int_{\Omega}|u|^{r(x)}dx\\
&= \Big(\frac{1}{p^{+}} - \frac{1}{r^{-}}\Big)\int_{\Omega}|\nabla u|^{p(x)}dx+\frac{1}{r^{-}}I(u)\\
&>\Big(\frac{1}{p^{+}} - \frac{1}{r^{-}}\Big)\int_{\Omega}|\nabla u|^{p(x)}dx.
\end{split}\end{equation*}
Therefore we have
\begin{align}\label{Key ine02}
\int_{\Omega}|\nabla u|^{p(x)}dx<\frac{sp^{+}r^{-}}{r^{-}-p^{+}}.
 \end{align}
The proof is complete.
\end{proof}

Next, we discuss the properties of $\lambda_{\alpha}$ and $\Lambda_{\alpha}.$
\begin{lemma}\label{LEM304}
For any $s>d,~ \lambda_{s}$ and $\Lambda_{s}$ defined in $(\ref{4})$ satisfy $0< \lambda_{s}<\Lambda_{s}<+\infty$.
\end{lemma}
\begin{proof}
For $u\in V$, $\hbox{Gagliardo-Nirenberg}$ inequality and the embedding
$L^{r^+}(\Omega)\hookrightarrow L^{r(x)}(\Omega)\hookrightarrow L^{r^{-}}(\Omega)$ indicate that there exists a positive constant $\tilde{C}$ such that\\
\begin{equation}\label{1}
  \|u\|_{r(x)}\leq (|\Omega|+1)\|u\|_{r^+}\leq \tilde{C}\|\nabla u\|^\theta_{p(x)}\|u\|_2^{1-\theta},
\end{equation}
where $\theta$ is determined by $\left(\frac{1}{2}+\frac{1}{N}-\frac{1}{p^-}\right)\theta=\frac{1}{2}-\frac{1}{r^+}$. It is easy to check that $\theta\in (0,1)$ due to condition $(\mathbf{H})$.

Obviously, due to $u\in\mathcal{N}$ we have $\int_\Omega|\nabla u|^{p(x)}dx=\int_\Omega|u|^{r(x)}dx$ and utilize Lemma \ref{PRO102} to obtain
\begin{equation}\label{10}
\begin{split}
\|u\|_{r(x)}&\geq \min\Big\{\Big(\int_{\Omega}| u|^{r(x)}dx\Big)^{\frac{1}{r^{+}}},
\Big(\int_{\Omega}|u|^{r(x)}dx\Big)^{\frac{1}{r^{-}}}\Big\}
\\&\geq \min\Big\{\Big(\int_{\Omega}|\nabla u|^{p(x)}dx\Big)^{\frac{1}{r^{+}}},
\Big(\int_{\Omega}|\nabla u|^{p(x)}dx\Big)^{\frac{1}{r^{-}}}\Big\}.
\end{split}
\end{equation}
Then, we combine (\ref{1}) with (\ref{10}) to get
\begin{equation}\label{addequation}
\begin{split}
\|u\|_{2}\geq\min
\Bigg\{\Big[\frac{1}{\tilde{C}}\Big(\int_{\Omega}|\nabla u|^{p(x)}dx\Big)^{\frac{1}{r^{+}}
-\frac{\theta}{p^{-}}}\Big]^{\frac{1}{1-\theta}},\Big[\frac{1}{\tilde{C}}\Big(\int_{\Omega}|\nabla u|^{p(x)}dx\Big)^{\frac{1}{r^{-}}
-\frac{\theta}{p^{+}}}\Big]^{\frac{1}{1-\theta}}\Bigg\}.
\end{split}\end{equation}
So, the right-hand side of the above inequality remains bounded away from $0$ no
matter what the sign of $\frac{1}{r^+}-\frac{\theta}{p^-}$ and $\frac{1}{r^-}-\frac{\theta}{p^+}$ are due to Inequalities $(\ref{Key ine01})$ and $(\ref{Key ine02})$. Therefore, $\lambda_{\alpha}>0.$

On the other hand, the fact that $\Lambda_{s}<\infty$ just follows from (\ref{Key ine02}) and the Sobolev embedding inequality (\ref{201}).
This completes the proof of this lemma.
\end{proof}
In fact, we can get a lower bound of $\lambda_s$ that is independent of $s$.
\begin{lemma}\label{LEM305}
If
$
2\leq r^+\leq(1+\frac{2}{N})p^{-},
$
then for any $s>d,$ $\lambda_s\geq M>0$, where
$M$ is given in \eqref{addequation1}.
 \end{lemma}
 \begin{proof}
Obviously, \eqref{addequation} still holds. \eqref{Key ine} implies
\[\int_{\Omega}|\nabla u|^{p(x)}dx\geq \frac{p^{-}r^{+}}{r^{+}-p^{-}}d.\]
It is obvious by combining $2\leq r^+\leq(1+\frac{2}{N})p^{-}$ with the above inequality and \eqref{addequation} that
\begin{equation}\label{addequation1}
\begin{split}
\|u\|_{2}\geq\min
\Bigg\{\Big[\frac{1}{\tilde{C}}\Big(\frac{p^{-}r^{+}}{r^{+}-p^{-}}d\Big)^{\frac{1}{r^{+}}
-\frac{\theta}{p^{-}}}\Big]^{\frac{1}{1-\theta}},\Big[\frac{1}{\tilde{C}}\Big(\frac{p^{-}r^{+}}{r^{+}-p^{-}}d\Big)^{\frac{1}{r^{-}}
-\frac{\theta}{p^{+}}}\Big]^{\frac{1}{1-\theta}}\Bigg\}:=M.
\end{split}\end{equation}
 \end{proof}
 The forthcoming lemma tells us that two sets $\mathcal{N}_{+}$ and $\mathcal{N}_{-}$ are invariant.
 \begin{lemma}\label{lem305}
Let  ${\mathbf{(H)}}$ hold and assume that $u(x, t)$ is a weak solution of Problem \eqref{equ101} in $\Omega\times[0,T) $ with $J(u_0)<d$.

\ \ $\mathrm{(i)}$ If $I(u_0)>0$, then $u(x, t)\in \mathcal{N}_{+}$ for $0<t<T$.

 \ \ $\mathrm{(ii)}$ If $I(u_0)<0$, then $u(x, t)\in \mathcal{N}_{-}$ for $0<t<T$.
\end{lemma}
\begin{proof}
$\mathrm{(i)}$ For $J(u_0)<d$, $I(u_0)>0$, from the definition of $\mathcal{N}_{+}$, we know $u_0\in \mathcal{N}_{+}$.
Next we will prove $u(t)\in \mathcal{N}_{+}$ for $0<t<T$. Otherwise, there exists
a $t_0\in(0, T)$ such that $u(t_0)\in\partial \mathcal{N}_{+}$.
Noticing that $0$ is an interior point of $\mathcal{N}_{+}$, we thus have
\begin{eqnarray*}
I(u(t_0))=0, \ \|\nabla u(t_0)\|_{p(.)}\neq0, \ \ \text{or} \ \ J(u(t_0))=d.
\end{eqnarray*}
As $J(u(t_0))<d$ by \eqref{3}, we thus have $I(u(t_0))=0$ and $\|\nabla u(t_0)\|_{p(.)}\neq0$, which, by the definition of $d$, implies that
$J(u(t_0))\geq d$, a contradiction to \eqref{3}.

\ \ \ \ $\mathrm{(ii)}$ Similarly, we have $u_0\in \mathcal{N}_{-}$. Next we will show that $u(t)\in \mathcal{N}_{-}$ for $0<t<T$. If not, there exists
a $t_0\in(0, T)$ such that $u(t_0)\in\partial \mathcal{N}_{-}$, namely
\begin{eqnarray*}
I(u(t_0))=0, \ \ \text{or} \ \ J(u(t_0))=d.
\end{eqnarray*}
By \eqref{3}, we can see that $J(u(t_0))<d$, then $I(u(t_0))=0$. We assume that $t_0$ is the first time such that
$I(u(t))=0$, then $I(u(t))<0$ for $0\leq t<t_0$. By (\ref{Key ine01}) we have $$\int_{\Omega}|\nabla u(.,t)|^{p(x)}dx>\alpha_1,~~
{\rm for}~~0\leq t<t_0.$$ Hence
\begin{align*}
\int_{\Omega}|\nabla u(.,t_{0})|^{p(x)}dx=\lim\limits_{t\rightarrow t_{0}}\int_{\Omega}|\nabla u(.,t)|^{p(x)}dx\geq\min\Big\{B^{\frac{r^+p^-}{p^--r^+}},B^{\frac{r^+p^+}{p^+-r^-}}\Big\}>0,
\end{align*} which together with $I(u(t_0))=0$ implies that $u(t_0)\in \mathcal{N}$.
By the definition of $d$, we again obtain $J(u(t_0))\geq d$, a contradiction to \eqref{3}. The proof is complete.
\end{proof}
\section {The case $r(x)\geq p(x)$}
In this section, we mainly discuss the behavior of the solution to Problem $\eqref{equ101}$. First of all, we modify the classical potential well method which first is introduced by Sattinger in \cite{DHS} and developed by Levine \cite{Lev74,Lev74-2}, and then apply energy estimate method to give threshold results for the solutions to exist globally or to blow up in finite time
 in the subcritical(critical) case $J(u_{0})<d$($J(u_{0})=d$). Subsequently,
together with Lemma \ref{lem304} and Lemma \ref{LEM304}, we can obtain an abstract criterion for the existence of global solutions that tend to $0$ as $t$ tends to $\infty$ or finite-time blow-up solutions in terms of $\lambda_{s}$ and $\Lambda_{s}$ for supercritical initial energy, $i.e.$ ~$J(u_0)>d$.
Our main results are as follows:
\begin{theorem}\label{th3.1}(Global existence for $J(u_0)<d$.)
Let  ${\mathbf{(H)}}$ hold and $u_0\in V(\Omega)$. If
$J(u_0)<d$ and $I(u_0)>0$, then Problem \eqref{equ101} admits a global weak solution
$u\in L^{\infty}(0,\infty:V(\Omega)),~u_{t}\in L^{2}(\Omega\times(0,\infty)),u(t)\in \mathcal{N}_{+}$ for $0\leq t<\infty$.
Moreover,
there exist two constants $K_{0},K_{1}>0$ defined in \eqref{3.35} such that
\begin{eqnarray*}
\int_{\Omega}|\nabla u|^{p(x)}dx\leq\begin{cases}\frac{p^{+}r^{-}}{r^{-}-p^{+}}\Big(\frac{K^{2}_{1}p^{+}}{K_{1}+K_{0}(p^{+}-2)t}\Big)^{\frac{p^{+}}{p^{+}-2}},~~&p^{+}>2;\\
\frac{p^{+}r^{-}}{r^{-}-p^{+}}K_{1}e^{\frac{1}{K_{0}}(K_{0}-t)},~~&p^{+}= 2,\end{cases}
\end{eqnarray*}
\end{theorem}
\begin{proof}
The proof is divided into two steps.\\
{\bf Step 1. Global existence.}

A weak solution $u(x,t)$ to Problem \eqref{equ101} is a limit function of the sequence of\break $\hbox{Galerkin's}$ approximation
$$u^{(m)}(x,t)=\sum\limits_{k=1}^{m}C_{k}^{(m)}(t)\varphi_{k}(x),~~\varphi_{k}(x)\in V(\Omega),~C_{k}^{(m)}(t)\in C^{1}[0,T),$$
where the coefficients $C^{m}_{k}(t)$ satisfy the relations
\begin{equation}\label{e202}
\begin{split}
&\int_{\Omega}[u^{(m)}_{t}\varphi_{k}+|\nabla
u^{(m)}|^{p(x)-2}\nabla u^{(m)}\nabla\varphi_{k}-F(x,t,u^{(m)})\varphi_{k}]dx=0,~~k=1,2,...m,\\
&~~~~~~~u^{(m)}(x,0)=\sum_{j=1}^m b^{(m)}_j\varphi_j(x)\rightarrow u_0(x) \ \ \text{in} \ \ V(\Omega).
\end{split}\end{equation}
Here $F(x,t,u^{(m)})=|u^{(m)}|^{r(x)-2}u^{(m)}
-\frac{1}{|\Omega|}\int_{\Omega}|u^{(m)}|^{r(x)-2}u^{(m)}dx.$

The existence of a local solution to system \eqref{e202} is guaranteed by Peano's theorem.
By the uniform estimates \eqref{e1}-\eqref{e3}, it is seen that the local solutions can be extended globally.

In fact, multiplying (\ref{e202}) by $(C_{k}^{(m)}(t))'$ and summing over
$k=1,2,\cdots,m,$ we arrive at the relation
\begin{eqnarray}\label{e301}
\int_0^t\|u^{(m)}_{\tau}\|_2^2\mathrm{d}\tau+J(u^{(m)})=J(u^{(m)}(0)), \ \ \ 0\leq t<\infty.
\end{eqnarray}
Due to the convergence of $u^m(x, 0)\rightarrow u_0(x)$ in $ V(\Omega)$, one has
\begin{eqnarray*}
J(u^{(m)}(x, 0))\rightarrow J(u_0(x))<d\quad\text{and}\quad I(u^{(m)}(x,0))\rightarrow I(u_0(x))>0.
\end{eqnarray*}
Therefore, for sufficiently large $m$ and for any $0\leq t<\infty$, we obtain
\begin{eqnarray}\label{3.4}
\int_0^t\|u^{(m)}_{\tau}\|_2^2\mathrm{d}\tau+J(u^{(m)})=J(u^{(m)}(0))<d\quad\text{and}\quad I(u^m(x,0))>0,
\end{eqnarray}
which implies that $u^m(x,0)\in \mathcal{N}_{+}$ for sufficiently large $m$.

By applying the similar arguments of Lemma \ref{lem305},
one can show from \eqref{3.4} that $u^m(x,t)\in \mathcal{N}_{+}$ for sufficiently large $m$ and $0\leq t<\infty$. Moreover, by the definition $\mathcal{N}_+$, we deduce $I(u^m(x,t))>0$ or $u^m(x,t)=0$. Then, using the following inequality
\begin{eqnarray*}
J(u^{(m)})\geq\frac{r^{-}-p^{+}}{p^{+}r^{-}}\int_{\Omega}|\nabla u^{(m)}|^{p(x)}dx+\frac{1}{r^{-}}I(u^{(m)}),
\end{eqnarray*}
and \eqref{3.4}, we obtain
\begin{eqnarray}\label{3.5}
\int_0^t\|u^{(m)}_{\tau}\|_2^2\mathrm{d}\tau+\frac{r^{-}-p^{+}}{p^{+}r^{-}}\int_{\Omega}|\nabla u^{(m)}|^{p(x)}dx<d,
\end{eqnarray}
for sufficiently large $m$ and for any $0\leq t<\infty$, which then yields
\begin{equation}\label{e1}
\int_{\Omega}|\nabla u^{(m)}|^{p(x)}dx\leq\frac{p^{+}r^{-}d}{r^{-}-p^{+}},\quad 0\leq t<\infty,
\end{equation}
\begin{equation}\label{e2}
\int_0^t\|u^{(m)}_{\tau}\|_2^2\mathrm{d}\tau<d,\quad 0\leq t<\infty,
\end{equation}
\begin{equation}\label{e3}
\|u^{(m)}\|_{r(x)}\leq B\|\nabla u^{(m)}\|_{p(.)}\leq B\max\Big\{\Big(\frac{p^{+}r^{-}d}{r^{-}-p^{+}}\Big)^{\frac{1}{p^{-}}},\Big(\frac{p^{+}r^{-}d}{r^{-}-p^{+}}\Big)^{\frac{1}{p^{-}}}\Big\},\quad 0\leq t<\infty.
\end{equation}
Once again, we follow the lines of the proof of Lemma \ref{lem301} to obtain that there exists
 $u(x,t)\in L^{\infty}(0,\infty, V(\Omega)),u_{t}\in L^{2}(\Omega\times(0,\infty))$ such that
 \begin{align}\label{addd104}
 J(u(t))+\int_{0}^{t}\int_{\Omega}|u_{\tau}|^{2}dxd\tau\leq J(u_{0}).
 \end{align}
Moreover, the first conclusion of Lemma \ref{lem305} together with (\ref{addd104}) tells us the fact $u(x,t)\in \mathcal{N}_{+},t>0.$

{\bf Step 2. Decay rate.} Choosing the test function $\varphi=u(x,t)$
 in \eqref{e201}, picking $t_{1}=t,~t_{2}=t+\delta,~(0<t<t+\delta<T)$, multiplying by $\frac{1}{\delta}$ and letting
$\delta\rightarrow0^+$, we obtain the relations by
$\textrm{Lebesgue}$ Dominated Convergence theorem that
\begin{eqnarray}\label{3.33}
\frac{1}{2}\frac{d}{dt}\|u\|_2^2=(u_t, u)=-\int_{\Omega}|\nabla u|^{p(x)}dx+\int_{\Omega}|u|^{r(x)}dx=-I(u).
\end{eqnarray}
From Lemma \ref{lem305} we know that $u(x, t)\in \mathcal{N}_{+}$ for $0<t<\infty$ under the condition
$J(u_0)<d$ and $I(u_0)>0$. Thus we have $I(u)>0$ for $u\not\equiv0,~0<t<\infty$. In addition,  the fact $\lim\limits_{\delta\rightarrow1^{-}}\Big(\delta\int_{\Omega}|\nabla u|^{p(x)}dx-\int_{\Omega}|u|^{r(x)}dx\Big)=I(u)$ ensures there must exist a $0<\delta_{0}<1$ such that
\begin{align}\label{3.34}
\delta_{0}\int_{\Omega}|\nabla u|^{p(x)}dx-\int_{\Omega}|u|^{r(x)}dx\geq0,~~0\not\equiv u\in \mathcal{N}_{+}.
\end{align}

To obtain  decay estimate of solutions, we first establish the equivalent relations between $J(u)$ and $I(u)$. Actually, on the one hand, by \eqref{3.34} and the definition of $I(u)$, we have
\begin{align}\label{add3.34}
I(u)>\int_{\Omega}|\nabla u|^{p(x)}dx-\delta_{0}\int_{\Omega}|\nabla u|^{p(x)}dx=(1-\delta_{0})\int_{\Omega}|\nabla u|^{p(x)}dx.
\end{align}
Further, the definition of $J(u)$ and \eqref{add3.34} yield
\begin{align}\label{add3.34-1}
J(u)\leq\frac{1}{p^{-}}\int_{\Omega}|\nabla u|^{p(x)}dx
\leq\frac{I(u)}{p^{-}(1-\delta_{0})}.
\end{align}
On the other hand, by means of the definition of $J(u)$ and $I(u)$ as well as $I(u)>0$, we have
\begin{equation}\label{add3.34-2}
\begin{split}
J(u)&\geq\frac{1}{p^{+}}\int_{\Omega}|\nabla u|^{p(x)}dx-\frac{1}{r^{-}}\int_{\Omega}|u|^{r(x)}dx
\\&=\frac{r^{-}-p^{+}}{p^{+}r^{-}}\int_{\Omega}|\nabla u|^{p(x)}dx+\frac{1}{r^{-}}I(u)
\geq\frac{r^{-}-p^{+}}{p^{+}r^{-}}\int_{\Omega}|\nabla u|^{p(x)}dx.
\end{split}
\end{equation}
Therefore, by \eqref{3.33} and \eqref{add3.34-1}, one has for $0\leqslant S<T<\infty,$
\begin{equation}\label{add3.34-3}
\begin{split}
\int_{S}^{T}J(u(t))dt&\leq\frac{1}{p^{-}(1-\delta_{0})}\int_{S}^{T}I(u(t))dt
=\frac{1}{2p^{-}(1-\delta_{0})}\int_{S}^{T}\frac{d}{dt}\|u\|_2^2dt
\\&\leq\frac{1}{2p^{-}(1-\delta_{0})}\|u(S)\|_2^2.
\end{split}
\end{equation}
Once again, we utilize the embedding inequality $\|u\|_{2}\leq B_{0}\|\nabla u\|_{p(.)}, \forall\ u\in V(\Omega)$ ($B_{0}>0$ is the best embedding constant from $V(\Omega)$ to $L^2(\Omega)$) to obtain
\begin{equation}\label{add3.34-4}
\begin{split}
\frac{1}{2p^{-}(1-\delta_{0})}\|u(S)\|_2^2&\leq\frac{B^{2}_{0}}{2p^{-}(1-\delta_{0})}\|\nabla u(S)\|^{2}_{p(.)}
\\
&\leq \frac{B^{2}_{0}}{2p^{-}(1-\delta_{0})}\max\Big\{(\int_{\Omega}|\nabla u(S)|^{p(x)}dx)^{\frac{2}{p^{+}}},
(\int_{\Omega}|\nabla u(S)|^{p(x)}dx)^{\frac{2}{p^{-}}}\Big\}
\\
&{\overset{\eqref{e1}}\leq}\frac{B^{2}_{0}}{2p^{-}(1-\delta_{0})}\max\Big\{1,
(\frac{p^{+}r^{-}d}{r^{-}-p^{+}})^{\frac{2}{p^{-}}-\frac{2}{p^{+}}}\Big\}
(\int_{\Omega}|\nabla u(S)|^{p(x)}dx)^{\frac{2}{p^{+}}}
\\
&\overset{\eqref{add3.34-2}}\leq\frac{B^{2}_{0}}{2p^{-}(1-\delta_{0})}\max\Big\{1,
(\frac{p^{+}r^{-}d}{r^{-}-p^{+}})^{\frac{2}{p^{-}}-\frac{2}{p^{+}}}\Big\}
\Big(\frac{p^{+}r^{-}}{r^{-}-p^{+}}\Big)^{\frac{2}{p^{+}}}
J^{\frac{2}{p^{+}}}(u(S))\\
&=K_{0}J^{\frac{2}{p^{+}}}(u(S)).
\end{split}
\end{equation}
where
\begin{align*}
K_{0}=\frac{B^{2}_{0}}{2p^{-}(1-\delta_{0})}\max\Big\{1,
(\frac{p^{+}r^{-}d}{r^{-}-p^{+}})^{\frac{2}{p^{-}}-\frac{2}{p^{+}}}\Big\}
\Big(\frac{p^{+}r^{-}}{r^{-}-p^{+}}\Big)^{\frac{2}{p^{+}}}.
\end{align*}
\eqref{add3.34-3} and \eqref{add3.34-4} indicate that
\begin{align*}
\int_{S}^{T}J(u(t))dt\leq K_{0}J^{\frac{2}{p^{+}}}(u(S)).
\end{align*}
Letting $T\rightarrow\infty$, we have
\begin{align}\label{add3.34-6}
\int_{S}^{\infty}J(u(t))dt\leq K_{0}J^{\frac{2}{p^{+}}}(u(S)).
\end{align}
Consequently, for $t>0$, the forthcoming estimates follow from \eqref{add3.34-6} and Lemma 1 of \cite{PM}
\begin{eqnarray*}
J(u(t))\leq\begin{cases}\Big(\frac{K^{2}_{1}p^{+}}{K_{1}+K_{0}(p^{+}-2)t}\Big)^{\frac{p^{+}}{p^{+}-2}},~~&p^{+}>2;\\
K_{1}e^{\frac{1}{K_{0}}(K_{0}-t)},~~&p^{+}= 2,\end{cases}
\end{eqnarray*}
where
\begin{align}\label{3.35}
K_{1}=J(u_{0}),~K_{0}=\frac{B^{2}_{0}}{2p^{-}(1-\delta_{0})}\max\Big\{1,
(\frac{p^{+}r^{-}d}{r^{-}-p^{+}})^{\frac{2}{p^{-}}-\frac{2}{p^{+}}}\Big\}
\Big(\frac{p^{+}r^{-}}{r^{-}-p^{+}}\Big)^{\frac{2}{p^{+}}}.
\end{align}
The proof is complete.
\end{proof}
\begin{theorem}\label{th3.2}(Blow-up for $J(u_0)<d$.)
Let  ${\mathbf{(H)}}$ hold and $u$ be a weak solution of Problem \eqref{equ101} with $u_0\in V(\Omega)$.
If $J(u_0)<d$ and $I(u_0)<0$, then there exists a finite time $T$ such that $u$ blows up at $T$.
\end{theorem}
\begin{proof}
Some ideas of the proof comes from Theorem 3.1 in \cite{YZH}.
We will divide the proof of this theorem into the following three steps:

{\bf Step 1. Establish a differential inequality for the solution.}
Assume on the contrary that $u$ is a global weak solution to Problem \eqref{equ101} with $J(u_0)<d$, $I(u_0)<0$ and define
\begin{eqnarray*}
M(t)=\int_0^t\|u\|_2^2\mathrm{d}\tau,\quad\ t\geq0,
\end{eqnarray*}
then
\begin{eqnarray}\label{3.11}
M'(t)=\|u\|_2^2,
\end{eqnarray}
and
\begin{eqnarray}\label{3.12}
M''(t)=2(u_t, u)=-2\Big(\int_{\Omega}|\nabla u|^{p(x)}dx-\int_{\Omega}|u|^{r(x)}dx\Big)=-2I(u).
\end{eqnarray}
Direct computations show that
\begin{eqnarray}\label{3.13}
J(u)\geq \frac{r^{-}-p^{+}}{r^{-}p^{+}}\int_{\Omega}|\nabla u|^{p(x)}dx+\frac{1}{r^{-}}I(u).
\end{eqnarray}
By \eqref{3}, \eqref{3.12} and \eqref{3.13}, we can get
\begin{eqnarray*}
M''(t)&\geq&\frac{2(r^{-}-p^{+})}{p^{+}}\int_{\Omega}|\nabla u|^{p(x)}dx-2r^{-}J(u)\\
&=&\frac{2(r^{-}-p^{+})}{p^{+}}\int_{\Omega}|\nabla u|^{p(x)}dx+2r^{-}\int_0^t\|u_{\tau}\|_2^2\mathrm{d}\tau-2r^{-}J(u_0).
\end{eqnarray*}
Noticing that
\begin{eqnarray*}
(M'(t))^2=4\bigg(\int_0^t\int_\Omega u_{\tau}u\mathrm{d}x\mathrm{d}\tau\bigg)^2+2\|u_0\|_2^2M'(t)-\|u_0\|_2^4,
\end{eqnarray*}
we have
\begin{eqnarray*}
M''(t)M(t)-\frac{r^{-}}{2}M'(t)^2&\geq&2r^{-}\int_0^t\|u_{\tau}\|_2^2\mathrm{d}\tau\int_0^t\|u\|_2^2\mathrm{d}\tau -2r^{-}J(u_0)M(t)\\
&&+\frac{2(r^{-}-p^{+})}{p^{+}}\int_{\Omega}|\nabla u|^{p(x)}dx M(t)-2r^{-}\bigg(\int_0^t\int_{\Omega}u_{\tau}u\mathrm{d}x\mathrm{d}\tau\bigg)^2\\
&&-r^{-}\|u_0\|^2_2M'(t)+\frac{r^{-}}{2}\|u_0\|_2^4.
\end{eqnarray*}
With the help of Cauchy-Schwarz inequality, one derives
$$\Big(\int_0^t\int_{\Omega}u_{\tau}u\mathrm{d}x\mathrm{d}\tau\Big)^2\leq \int_0^t\|u_{\tau}\|_2^2\mathrm{d}\tau\int_0^t\|u\|_2^2\mathrm{d}\tau.$$
we further obtain
\begin{equation}\label{3.14}\begin{split}
M''(t)M(t)-\frac{r^{-}}{2}M'(t)^2&\geq
\Bigg(\frac{2(r^{-}-p^{+})}{p^{+}}\int_{\Omega}|\nabla u|^{p(x)}dx-2r^{-}J(u_0)\Bigg) M(t)-\ r^{-}\|u_0\|^2_2M'(t).
\end{split}
\end{equation}
Next, we establish the quantitative relationship between $\int_{\Omega}|\nabla u|^{p(x)}dx$ and $M'(t)$. Obviously, Lemma \ref{PRO102} and
Inequality \eqref{201} yield
\begin{equation}\label{3.15}\begin{split}
\int_{\Omega}|\nabla u|^{p(x)}dx
&\geq\min\Big\{\|\nabla u\|^{p{+}}_{p(.)},\|\nabla u\|^{p{+}}_{p(.)}\Big\}
\geq\min\Big\{(B_{0}^{-1}\|u\|_{2})^{p{+}},(B_{0}^{-1}\|u\|_{2})^{p{-}}\Big\}
\\&=\min\Big\{(B_{0}^{-2}M'(t))^{\frac{p{+}}{2}},(B_{0}^{-2}M'(t))^{\frac{p{-}}{2}}\Big\}.
\end{split}
\end{equation}
Hence, by combining \eqref{3.14} with \eqref{3.15}, we have
\begin{equation}\label{3.16}
\begin{split}
M''(t)M(t)-\frac{r^{-}+2}{4}M'(t)^2&\geq
\frac{2(r^{-}-p^{+})}{p^{+}}\min\Big\{(B_{0}^{-2}M'(t))^{\frac{p{+}}{2}},(B_{0}^{-2}M'(t))^{\frac{p{-}}{2}}\Big\}M(t)
\\&-2r^{-}J(u_0) M(t)+\frac{r^{-}-2}{4}M'(t)^2-\ r^{-}\|u_0\|^2_2M'(t).
\end{split}
\end{equation}
{\bf Step 2. Important observation.} For~any~$u\in \mathcal{N}_{-}$, we claim
\begin{align}\label{3.17}
I(u)\leq r^{-}(J(u)-d).
\end{align}To prove this assertion, first of all,
 for any $\lambda>0$, define $f(\lambda)=I(\lambda u),$ then it is not hard to verify that $f(\lambda)$ is continuous with respect to $\lambda$.  A simple computation shows that
\begin{equation*}
f(\lambda)\geq \lambda^{p^{+}}\int_{\Omega}|\nabla u|^{p(x)}dx-\lambda^{r^{-}}\int_{\Omega}|u|^{r(x)}dx,~~0<\lambda<1.
\end{equation*}
On the one hand, Lemma \ref{lem305} ensures $u(t)\in \mathcal{N}_{-},$ for any $t\geq0$. i.e. $I(u)<0,~~t>0.$
On the other hand, from $r^{-}>p^{+}$, it is not difficult to deduce that there exists a positive constant $0<\lambda_{0}:=\lambda_{0}(u)<1$ such that
\begin{align*}
\lambda_{0}^{p^{+}}\int_{\Omega}|\nabla u|^{p(x)}dx-\lambda_{0}^{r^{-}}\int_{\Omega}|u|^{r(x)}dx>0\Rightarrow f(\lambda_{0})>0,
\end{align*}
which yields that, from the intermediate value theorem and $f(1)=I(u)<0$, there exists a positive constant $1>\lambda^{*}>0$ such that $f(\lambda^{*})=0,$ i.e. $I(\lambda^{*}u)=0.$

Subsequently, set $g(\lambda)=r^{-}J(\lambda u)-f(\lambda).$ Then, for any $0\not\equiv u\in  V(\Omega)$, a direct computation tells us that
\begin{align*}
g'(\lambda)&=r^{-}J'(\lambda u)-f'(\lambda)=r^{-}\Big(\int_{\Omega}\lambda^{p(x)-1}|\nabla u|^{p(x)}dx-\int_{\Omega}\lambda^{r(x)-1}|u|^{r(x)}dx\Big)\\
&~~~~-
\Big(\int_{\Omega}p(x)\lambda^{p(x)-1}|\nabla u|^{p(x)}dx-\int_{\Omega}r(x)\lambda^{r(x)-1}|u|^{r(x)}dx\Big)\\
&\geq\int_{\Omega}(r^{-}-p(x))\lambda^{p(x)-1}|\nabla u|^{p(x)}dx\geq0,
\end{align*}
which implies that
\begin{align*}
g(1)\geq g(\lambda^{*})\Rightarrow r^{-}J( u)-I(u)\geq r^{-}J(\lambda^{*} u)-I(\lambda^{*}u).
\end{align*}
It is obvious that \eqref{3.17} follows from the inequality above and the definition of $d$.

{\bf Step 3. Derive a contradiction.}
\begin{eqnarray*}
M''(t)=-2I(u)\geq2r^{-}(d-J(u))\geq2r^{-}(d-J(u_{0}))>0,
\end{eqnarray*}
which in turn implies for all $t\geq 0$ that
\begin{eqnarray*}
&&M'(t)\geq 2r^{-}(d-J(u_{0}))t, \\
&&M(t)\geq r^{-}(d-J(u_{0}))t^2.
\end{eqnarray*}
Therefore, for sufficiently large $t$, we have
\begin{eqnarray*}
&&\frac{2(r^{-}-p^{+})}{p^{+}}\min\Big\{(B_{0}^{-2}M'(t))^{\frac{p{+}}{2}},(B_{0}^{-2}M'(t))^{\frac{p{-}}{2}}\Big\}
-2r^{-}J(u_0)\geq0,\\
&&\frac{r^{-}-2}{4}M'(t)-\ r^{-}\|u_0\|^2_2\geq0.
\end{eqnarray*}

Consequently, from \eqref{3.14} and the above two inequalities, one has that
\begin{eqnarray}\label{3.18}
M''(t)M(t)-\frac{r^{-}+2}{4}M'(t)^2>0,~~t\geq t^{*},
\end{eqnarray}
where
\begin{align*}
t^{*}=\max\left\{\frac{1}{d-J(u_{0})}\frac{1}{2r^{-}}
\Bigg(\frac{p^{+}r^{-}|J(u_{0})|}{(r^{-}-p^{+})\min\Big\{B_{0}^{-p^{+}},B_{0}^{-p^{-}}\Big\}}\Bigg)^{\frac{2}{p^{-}}},
\frac{2\|u_{0}\|_{2}}{\sqrt{(r^{-}-2)(d-J(u_{0}))}}\right\}.
\end{align*}
From \eqref{3.18} it follows with $\alpha=\dfrac{r^{-}-2}{4}>0$ that
\begin{equation*}
\dfrac{\mathrm{d}}{\mathrm{d}t}\Big(\dfrac{M'(t)}{M^{1+\alpha}(t)}\Big)>0, \ \forall\ t\geq t^*,
\end{equation*}
or
\begin{equation}\label{3.19}
\dfrac{M'(t)}{M^{1+\alpha}(t)}>\dfrac{M'(t^*)}{M^{1+\alpha}(t^*)},\ \forall\ t>t^*.
\end{equation}
Integrating both sides of \eqref{3.19} over $(t^*,t)$ we can see that $M(t)$ can not remain finite for all $t>t^*$,
and therefore reaches a contradiction. Furthermore, By continuation theorem (in fact, we may follow the lines of the proof of  Theorem 3.1 of \cite{Bal77}), we have $\lim\limits_{t\rightarrow T}M(t)=+\infty.$ The proof is complete.
\end{proof}

\begin{remark} (Sharp condition for $J(u_0)<d$.)
Let (H) hold and $u_0\in  V(\Omega)$. Assume that $J(u_0)<d$.
If $I(u_0)>0$, Problem \eqref{equ101} admits a global weak solution; if $I(u_0)<0$,
all solutions to Problem \eqref{equ101} blow up in finite time.
\end{remark}
\begin{remark}
{\rm Noting that $B_{1}=B+1,$   a simple computation ensures that}
\begin{align*}
\frac{r^{-}-p^{+}}{p^{+}r^{-}}\min\Big\{B^{\frac{r^+p^-}{p^--r^+}},B^{\frac{r^-p^+}{p^+-r^-}}\Big\}>\Big(\frac{r^{-}-p^{+}}{p^{+}r^{-}}\Big)\alpha_1:= E_1,
\end{align*}
{\rm which implies that $E_{1}$ is not optimal. }
\end{remark}
For the critical case $J(u_0)=d$, the invariance of $\mathcal{N}_{+}$ can not be proved in general.
By using the method of approximation, we can still prove the global existence of weak solutions.
\begin{theorem}\label{th4.1}(Global existence for $J(u_0)=d.$)
Let  ${\mathbf{(H)}}$ hold and $u_0\in  V(\Omega)$. If $J(u_0)=d$ and $I(u_0)\geq0$, then Problem \eqref{equ101} admits a global weak
solution.
\end{theorem}
\begin{proof}
Let $\lambda_k=1-\frac{1}{k}$, $k=1, 2,\cdots$. Consider the following initial boundary value problem
\begin{eqnarray}\label{4.1}
\ \ \  \begin{cases}
u_t-\mathrm{div}(|\nabla u|^{p(x)-2}\nabla u)=|u|^{r(x)-2}u-\dfrac{1}{|\Omega|}\int_{\Omega}|u|^{r(x)-2}u dx, & (x,t)\in\Omega\times(0,T),\\
\dfrac{\partial u(x,t)}{\partial \nu}=0, & (x,t)\in\partial \Omega\times(0,T),\\
u(x,0)=\lambda_k u_0(x)\triangleq u_0^k, &x\in \Omega.
\end{cases}
\end{eqnarray}
Noticing that $I(u_0)\geq0$ and following the proof of \eqref{Add1001}, we can deduce that there exists a  $\lambda^{**}=\lambda^{**}(u_0)\geq1$ such that $I(\lambda^{**} u_0)=0$.
Then from $\lambda_k<1\leq\lambda^{**}$, we get $I(u_0^k)=I(\lambda_k u_0)>0$
and $J(u_0^k)=J(\lambda_k u_0)<J(u_0)=d$. In view of Theorem \ref{th3.1}, it follows that for each $k$ Problem \eqref{4.1}
admits a global weak solution $u\in L^{\infty} (0,\infty;L^{2}(\Omega))\cap  L^{p^{-}}(0,\infty; V(\Omega))$ with $u_t\in L^{2}(\Omega\times(0,\infty))$ and $u(t)\in \mathcal{N}_{+}$ for $0\leq t<\infty$ satisfying
\begin{eqnarray*}
\int_0^t\|u^k_{\tau}\|_2^2\mathrm{d}\tau+J(u^k)=J(u_0^k)<d.
\end{eqnarray*}
Applying the arguments similar to those in Theorem \ref{th3.1} we see that there exist a subsequence of $\{u^k\}$ and a function $u$, such
that $u$ is a weak solution of Problem \eqref{equ101} with $I(u)\geq0$ and $J(u)\leq d$
for $0\leq t<\infty$.
\end{proof}

\begin{theorem}\label{th4.2}
(Blow-up for $J(u_0)=d$.)
Let  ${\mathbf{(H)}}$ hold and $u$ be a weak solution of Problem \eqref{equ101} with $u_0\in  V(\Omega)$.
If $J(u_0)=d$ and $I(u_0)<0$, then there exists a finite time $T$ such that $u$ blows up at $T$.
\end{theorem}
\begin{proof}
Similarly to the proof of Theorem \ref{th3.2}, we can get
\begin{align}\label{4.2}
M''(t)M(t)-\frac{r^{-}+2}{4}M'(t)^2&\geq
\frac{2(r^{-}-p^{+})}{p^{+}}\min\Big\{(B_{0}^{-2}M'(t))^{\frac{p{+}}{2}},(B_{0}^{-2}M'(t))^{\frac{p{-}}{2}}\Big\}M(t)
\nonumber\\&-2r^{-}J(u_0) M(t)+\frac{r^{-}-2}{4}M'(t)^2-\ r^{-}\|u_0\|^2_2M'(t).
\end{align}
Since $J(u_0)=d$, $I(u_0)<0$, by the continuity of $J(u)$ and $I(u)$ with respect to $t$, there exists a $t_0>0$
such that $J(u(x, t))>0$ and $I(u(x,t))<0$ for $0<t\leq t_0$. From $(u_t, u)=-I(u)$, we have $u_t\not\equiv0$
for $0<t\leq t_0$. Furthermore, we have
\begin{eqnarray*}
J(u(t_0))\leq d-\int_0^{t_0}\|u_{\tau}\|_2^2\mathrm{d}\tau=d_1<d.
\end{eqnarray*}
Taking $t=t_0$ as the initial time and by Lemma \ref{lem305} (ii), we know that $u(x, t)\in \mathcal{N}_{-}$
for $t>t_0$.
The reminder of the proof is almost the same as that of Theorem \ref{th3.2} and hence is omitted.
\end{proof}

\begin{remark} (Sharp condition for $J(u_0)=d$.)
Let  ${\mathbf{(H)}}$ hold and $u_0\in  V(\Omega)$. Assume that $J(u_0)=d$.
If $I(u_0)\geq0$, Problem \eqref{equ101} admits a global weak solution; if $I(u_0)<0$,
Problem \eqref{equ101} admits no global weak solution.
\end{remark}

At the last part of this section, we can give an abstract criterion for the existence of global solutions that tend to $0$ as $t$ tends to $\infty$ or finite-time blow-up solutions in terms of $\lambda_{s}$ and $\Lambda_{s}$ for supercritical initial energy, $i.e.~J(u_0)>d$. Some ideas of the the following proof come from \cite{Gazzola}. Our main results are as follows:
\begin{theorem}\label{THM401}
Assume that $d<J(u_0)=E(0)\leq \alpha$ with $\alpha\in(d,+\infty),$ we have\begin{enumerate}[$(1)$]
\item~if $u_{0}\in N_{+}$ and $\|u_{0}\|\leq \lambda_{J(u_{0})}$, then $u_{0}\in \mathcal{G}$. That is, the weak solution $u$ of Problem \eqref{equ101} exists globally and $u(t)\to 0$ as $t\to +\infty,$
\item~if $u_{0}\in N_{-}$ and $\|u_{0}\|\leq \Lambda_{J(u_{0})}$, then $u_{0}\in \mathcal{U}$. That is, the weak solution $u$ of Problem \eqref{equ101} blows up in finite time.
\end{enumerate}\end{theorem}
\begin{proof}
Denote by $T(u_0)$ the maximal existence time of weak solutions for Problem (\ref{equ101}).  If there exists a global solution, i.e.  $T(u_0)=\infty$, we denote by \[\omega(u_0)=\bigcap\limits_{t\geq  0}\overline{\{u(s):s\geq  t\}}\]
the $\omega-$limit of $u_0$.

$(1)$ If $u_0\in \mathcal{N}_+$ and $\|u_0\|^2_{2}\leq \lambda_{J(u_0)}$, then we claim that $u\in\mathcal{N}_+$ for all $t\in [0,T(u_0))$. By contradiction, there exists a $t_0\in (0,T(u_0))$ such that $u\in\mathcal{N}_+$ for $t\in [0,t_0)$ and $u(t_0)\in \mathcal{N}$. Recalling \eqref{3.33}, we obtain
\begin{equation}
\label{2.17}
\frac{d}{dt}\|u\|_{2}^2=-2I(u).
\end{equation}
Therefore, it follows from \eqref{2.17} that $\int_0^t\|u_s\|_{2}^2ds\neq 0$ for $\Omega\times (0,t_0)$.
Further,  Lemma \eqref{lem301} yields $J(u(t_0))=E(t_0)<J(u_0)$, which implies $u(t_0)\in J^{J(u_0)}$. Therefore,  $u(t_0)\in \mathcal{N}_{J(u_0)}$. By the definition of $\lambda_{J(u_0)}$, we obtain
\begin{equation}
\label{2.18}
\|u(t_0)\|^2_{2}\geq \lambda_{J(u_0)}.
\end{equation}
Noticing that $I(u)>0$ for $t\in [0,t_0)$, it follows from (\ref{2.17}) that
$ \|u(t_0)\|^2_{2}<\|u_0\|^2_{2}\leq \lambda_{J(u_0)},$
which contradicts (\ref{2.18}). Therefore, $u\in\mathcal{N}_+$ for all $t\in [0,T(u_0))$. Further, we get  for all $t\in [0,T(u_0))$,
$u(t)\in J^{J(u_0)}\cap \mathcal{N}_+.$  Lemma \ref{lem304}(1) shows that the orbit ${u(t)}$ remains bounded in $V(\Omega)$  for $t\in [0,T(u_0))$ so that $T(u_0)=+\infty$.

For any $\omega\in\omega(u_0)$,  then
$\|\omega\|^2_{2}<\|u_0\|^2_{2}\leq \lambda_{J(u_0)},~J(\omega)\leq J(u_0)$
by (\ref{3}) and (\ref{2.17}). The second inequality implies $\omega\in J^{J(u_0)}$. Noticing that the definition of $\lambda_{J(u_0)}$ and  the first inequality, we derive $\omega\not\in \mathcal{N}_{J(u_0)}$, further $\omega\not\in \mathcal{N}$. Thus, $\omega(u_0)\cap \mathcal{N}=\emptyset$, which indicates $\omega(u_0)=\{0\}$. Therefore, the weak solution $u$ of Problem \eqref{equ101} exists globally and $u(t)\to 0$ as $t\to +\infty.$

(2) If $u_0\in \mathcal{N}_-$ and $\|u_0\|^2_{2}\geq \Lambda_{J(u_0)}$, then we claim that $u\in\mathcal{N}_-$ for all $t\in [0,T(u_0))$. By contradiction, there exists a $t^0\in (0,T(u_0))$ such that $u\in\mathcal{N}_-$ for $t\in [0,t^0)$ and $u(t^0)\in \mathcal{N}$.
Similar to case (1), we get $J(u(t^0))<J(u_0)$, which implies $u(t^0)\in J^{J(u_0)}$. Therefore,  $u(t^0)\in \mathcal{N}_{J(u_0)}$. By the definition of $\Lambda_{J(u_0)}$, we obtain
\begin{equation}
\label{2.19}
\|u(t^0)\|^2_{2}\leq  \Lambda_{J(u_0)}.
\end{equation}
Noticing that $I(u(t))<0$ for $t\in [0,t^0)$, it follows from (\ref{2.17}) that
 $\|u(t^0)\|^2_{2}>\|u_0\|^2_{2}\geq \Lambda_{J(u_0)},$
which contradicts (\ref{2.19}).  Therefore, $u\in\mathcal{N}_-$ for all $t\in [0,T(u_0))$. Further, we get  for all $t\in [0,T(u_0))$,
 $u(t)\in J^{J(u_0)}\cap \mathcal{N}_-.$
Suppose $T(u_0)=\infty$, then for any $\omega\in\omega(u_0)$,
$\|\omega\|^2_{2}> \|u_0\|^2_{2}\geq \Lambda_{J(u_0)},~J(\omega)\leq J(u_0)$
by (\ref{3}) and (\ref{2.17}).
The second inequality implies $\omega\in J^{J(u_0)}$. Noting that the definition of $\Lambda_{J(u_0)}$ and  the first inequality, we derive $\omega\not\in \mathcal{N}_{J(u_0)}$, further $\omega\not\in \mathcal{N}$. Thus, $\omega(u_0)\cap \mathcal{N}=\emptyset$, which indicates $\omega(u_0)=\{0\}$. This result contradicts  Lemma \ref{lem304}(1).
 Therefore, $\omega(u_0)=\emptyset$, $T(u_0)<\infty$, that is, the weak solution $u$ of problem \eqref{equ101} blows up in finite time.
\end{proof}

For $r(x)=r,$  in particular, we have the following Proposition \ref{PRO401} that is easily proved based on Theorem \ref{THM401}(2), here we omit the proof.  Meanwhile, we give the following Theorem \ref{THM402} to illustrate there exists $u_0$ such that $J(u_0$) is arbitrarily large, and the corresponding
solution $u(x, t)$ to Problem \eqref{equ101} with $u_0$ as initial datum blows up in finite time as well.
\begin{proposition}\label{PRO401}
Let $u_{0}\in  V(\Omega)$ with $J(u_{0})>d$. If $\frac{p^+r}{r-p^+}|\Omega|^{\frac{r-2}{2}}J(u_{0})\leq\|u_{0}\|_{2}^{r}$, then $u_{0}\in \mathcal{N}_{-}\cap \mathcal{U}$.
\end{proposition}

\begin{theorem}\label{THM402}
For any $M>d$, there exists a $u_{M}\in \mathcal{N}_{-}\cap B$ such that $J(u_{M})\geq M$.
\end{theorem}
\begin{proof}
For any $M>d$, let $\Omega_{1}$ and $\Omega_{2}$ be two arbitrary disjoint open subdomains of $\Omega$, and assume that $v\in  V(\Omega_{1})$ is an arbitrary nontrivial function. Since $r>p^+$, we can choose $\alpha >0$ large enough such that $J(\alpha v)\leq 0$ and $\|\alpha v\|_{2}^{r}>|\Omega|^{\frac{r-2}{2}}\frac{p^{+}r}{r-p^{+}}M$.

Fix $\alpha$ and choose a function $\omega\in  V(\Omega_{2})$ such that $J(\omega)+J(\alpha v)=M$. Extend $v$ and $\omega$ to be $0$ in $\Omega\setminus\Omega_{1}$ and $\Omega\setminus\Omega_{2}$, respectively, and set $u_{M}=\alpha v+\omega$. Then $J(u_{M})=J(\alpha v)+J(\omega)=M$, and
\begin{align*}
\|u_{M}\|_{2}^{r}\geq\|\alpha v\|_{2}^{r}>|\Omega|^{\frac{r-2}{2}}\frac{p^{+}r}{r-p^{+}}J(u_{M}).
\end{align*}
By Proposition \ref{PRO401}, it is seen that $u_{M}\in \mathcal{N}_{-}\cap \mathcal{U}$. The proof is complete.
\end{proof}

\section{The case $r(x)<p(x)$}
This section is devoted to the discussion of the behavior of solutions to Problem $\eqref{equ101}$ under the assumption that $r(x)<p(x),x\in\Omega.$ In such case,  it seems that the method used in the previous section is no longer effective owing to the negativity of the potential depth $d$. So, we have to search some new methods to study the properties of solutions.
Before stating our main results, we give the following key lemmas.
\begin{lemma}\label{lem500}
Suppose that $\alpha\geq\beta>0,~C_{2}\geq C_{1}>0$ and $h(t)$ is a nonnegative
and absolutely continuous function satisfying
\begin{align}\label{add500}
h'(t)+C_{1} \min\{h^{\alpha}(t),h^{\beta}(t)\}\leq C_{2},~~for~~t\in(0,\infty).\end{align} Then

$(1)$ if $h(0)\leq\Big(\frac{C_{2}}{C_{1}}\Big)^{\frac{1}{\beta}}$, then for~$t\geq0$, the following estimate is satisfied
\begin{align}\label{500}
h(t)\leq\displaystyle\Big(\frac{C_{2}}{C_{1}}\Big)^{\frac{1}{\beta}};\end{align}

$(2)$ if $h(0)>\Big(\frac{C_{2}}{C_{1}}\Big)^{\frac{1}{\beta}}$, then for~$t\geq0$, the following estimates are fulfilled
\begin{equation}\label{500-1}
h(t)\leq\begin{cases}
\Big(\frac{C_{2}}{C_{1}}\Big)^{\frac{1}{\beta}}+\Big[\Big(h(0)-\Big(\frac{C_{2}}{C_{1}}\Big)^{\frac{1}{\beta}}\Big)^{1-\beta}+C_{1}(\beta-1)t\Big]^{\frac{1}{1-\beta}},&~~\beta>1;\\
\frac{C_{2}}{C_{1}}h^{1-\beta}(0)+h(0)\Big(1-\frac{C_{2}h^{-\beta}(0)}{C_{1}}\Big)e^{-C_{1}h^{\beta-1}(0)t},&~~0<\beta\leq1.
\end{cases}
\end{equation}
\end{lemma}
\begin{proof} Some ideas come from \cite{RT}. Next, we complete this proof.

{\bf Case 1.} If $h(0)\leq
\Big(\frac{C_{2}}{C_{1}}\Big)^{\frac{1}{\beta}},$ then we have the following
claim
\begin{align}\label{501} h(t)\leq
\Big(\frac{C_{2}}{C_{1}}\Big)^{\frac{1}{\beta}},~\mathrm{for}~t\geq0. \end{align}

If not, then there would exist a $t_{0}>0$ and a $t_{1}>0$ such that
\begin{align}\label{502} h(t_{0})=
\Big(\frac{C_{2}}{C_{1}}\Big)^{\frac{1}{\beta}},~h(t)>
\Big(\frac{C_{2}}{C_{1}}\Big)^{\frac{1}{\beta}},~\mathrm{for}~t_{0}<t<t_{1}. \end{align}
Thus there exists a $t_{2}\in(t_{0},t_{1})$ such that
\begin{align}\label{503}
h'(t_{2})=\frac{h(t_{1})-h(t_{0})}{t_{1}-t_{0}}>0.\end{align}
On the other hand, noting that $h(t)>
(\frac{C_{2}}{C_{1}})^{\frac{1}{\beta}}\geq\Big(\frac{C_{2}}{C_{1}}\Big)^{\frac{1}{\alpha}}~\mathrm{for}~t_{0}<t<t_{1}$, we have
$$C_{1} \min\{h^{\alpha}(t),h^{\beta}(t)\}>C_{2}\Longrightarrow h'(t)<0,~\mathrm{for}~t_{0}<t<t_{1}.$$  This is a
contradiction.

{\bf Case 2.} If $h(0)>\Big(\frac{C_{2}}{C_{1}}\Big)^{\frac{1}{\beta}},$
then we claim that either
\begin{align}\label{504} h(t)\geq
\Big(\frac{C_{2}}{C_{1}}\Big)^{\frac{1}{\beta}},~\mathrm{for}~t\geq0,\end{align}
 or there
exists a $t^{*}>0$ such that
\begin{align}\label{505}
h(t)>\Big(\frac{C_{2}}{C_{1}}\Big)^{\frac{1}{\beta}},~~\mathrm{for}~0<t<t^{*};~
h(t)\leq\Big(\frac{C_{2}}{C_{1}}\Big)^{\frac{1}{\beta}},~~\mathrm{for}~t\geq
t^{*}.
\end{align}
If \eqref{504} holds, we have $ h(t)\geq
\Big(\frac{C_{2}}{C_{1}}\Big)^{\frac{1}{\beta}}\geq\Big(\frac{C_{2}}{C_{1}}\Big)^{\frac{1}{\alpha}},~\mathrm{for}~t\geq0.$
So
\begin{align*}\label{506}
h'(t)\leq C_{2}-C_{1} \min\{h^{\alpha}(t),h^{\beta}(t)\}<0\Longrightarrow
h(t)\leq h(0),~\mathrm{for}~t\geq0.
\end{align*}
If \eqref{505} holds, the proof of this part is analogous to that in Case
1 and \eqref{504}. In summary, we have
\begin{equation}\label{add506}
h(t)\leq h(0),~\mathrm{for}~t\geq0.
\end{equation}

Next, we will give decay estimate for $h(0)>\Big(\frac{C_{2}}{C_{1}}\Big)^{\frac{1}{\beta}}$. First, we consider the case $\beta>1.$ In this case, we claim that for~$t\geq0$, the inequality
\begin{equation}\label{507}
h(t)\leq\Big(\frac{C_{2}}{C_{1}}\Big)^{\frac{1}{\beta}}+\Big[\Big(h(0)-\Big(\frac{C_{2}}{C_{1}}\Big)^{\frac{1}{\beta}}\Big)^{1-\beta}+C_{1}(\beta-1)t\Big]^{\frac{1}{1-\beta}}
\end{equation} holds.
 As a matter of fact, if \eqref{504} holds, then we have $C_{1}h^{\beta}(t)\geq C_{2}$~for~$t\geq0$. So we set $z(t)=h(t)-\Big(\frac{C_{2}}{C_{1}}\Big)^{\frac{1}{\beta}}\geq0, \mathrm{for}~t\geq0.$ And then,
we apply the inequality $(a+b)^{\beta}\geq a^{\beta}+b^{\beta},~a,b\geq0,~\beta>1$ to obtain the following inequalities
\begin{align*}
z'(t)+C_{1}z^{\beta}(t)+C_{2}&\leq z'(t)+C_{1}\Bigg(z(t)+\Big(\frac{C_{2}}{C_{1}}\Big)^{\frac{1}{\beta}}\Bigg)^{\beta}\\
&=h'(t)
+C_{1} h^{\beta}(t)= h'(t)+C_{1} \min\{h^{\alpha}(t),h^{\beta}(t)\}\leq C_{2},
\end{align*}
which implies that
\begin{align}\label{508}
z'(t)+C_{1} z^{\beta}(t)\leq0.
\end{align}
Combining \eqref{508} with $\hbox{Gronwall's}$ inequality, we have
\begin{align}\label{509}
z(t)\leq\Big(z^{1-\beta}(0)+C_{1}(\beta-1)t\Big)^{\frac{1}{1-\beta}},
\end{align}
which ensures that the first inequality of \eqref{500-1} is true.

If \eqref{505} holds, for $0<t<t^{*}$, we follow the lines of the proof of \eqref{508} to get
\begin{align}\label{510}
z(t)\leq \Big(z^{1-\beta}(0)+C_{1}(\beta-1)t\Big)^{\frac{1}{1-\beta}},~\mathrm{for}~0<t\leq t^{*}.
\end{align}
For $t\geq  t^{*}$, it is obvious that
\begin{align*}
h(t)\leq\Big(\frac{C_{2}}{C_{1}}\Big)^{\frac{1}{\beta}}+\Big[\Big(h(0)-\Big(\frac{C_{2}}{C_{1}}\Big)^{\frac{1}{\beta}}\Big)^{1-\beta}+C_{1}(\beta-1)t\Big]^{\frac{1}{1-\beta}}.
\end{align*}
Immediately, the first inequality of \eqref{500-1} follows from two inequalities above.

Finally, we discuss the case $0<\beta\leq 1$. It is evident that \eqref{add506} ensures
 \begin{align}\label{511}
 \min\Big\{h^{\alpha}(t),h^{\beta}(t)\Big\}=h^{\beta}(t)\geq h^{\beta-1}(0)h(t).
 \end{align}
Moreover, by \eqref{511}, it is not hard to check that \eqref{add500} is equivalent to
 \begin{align}\label{512}
h'(t)+C_{1}h^{\beta-1}(0)h(t)\leq C_{2},~\mathrm{for}~t\geq0,
 \end{align}
 which shows that
 \begin{align}\label{513}
 h(t)\leq\frac{C_{2}}{C_{1}}h^{1-\beta}(0)+h(0)\Big(1-\frac{C_{2}h^{-\beta}(0)}{C_{1}}\Big)e^{-C_{1}h^{\beta-1}(0)t},~\mathrm{for}~t\geq0.
 \end{align}
\end{proof}
For $C_{1}>C_{2},$  we also have similar results as Lemma \ref{lem500}.
\begin{lemma}\label{lem501}
Suppose that $\alpha\geq \beta>0,~C_{1}> C_{2}>0$ and $h(t)$ is a nonnegative
and absolutely continuous function satisfying
\begin{align}\label{addlem501}
h'(t)+C_{1} \min\{h^{\alpha}(t),h^{\beta}(t)\}\leq C_{2},~\mathrm{for}~t\in(0,\infty).\end{align} Then

$(1)$ if $h(0)\leq\Big(\frac{C_{2}}{C_{1}}\Big)^{\frac{1}{\alpha}},$ then for~$t\geq0$, the following inequality is satisfied
\begin{align}\label{lem501-1}
h(t)\leq\displaystyle\Big(\frac{C_{2}}{C_{1}}\Big)^{\frac{1}{\alpha}};\end{align}

$(2)$ if $h(0)>
\Big(\frac{C_{2}}{C_{1}}\Big)^{\frac{1}{\alpha}},$ then for~$t\geq0$, the following estimate is fulfilled
\begin{align}\label{lem501-2}
h(t)\leq
\begin{cases}\Big(\frac{C_{2}}{C_{1}}\Big)^{\frac{1}{\alpha}}+\Big[\Big(h(0)-\Big(\frac{C_{2}}{C_{1}}\Big)^{\frac{1}{\alpha}}\Big)^{1-\beta}
+C_{1}\Big(\frac{C_{2}}{C_{1}}\Big)^{\frac{\alpha-\beta}{\alpha-\beta+1}}(\beta-1)t\Big]^{\frac{1}{1-\beta}},&\beta>1;\\
\big(\frac{C_{2}}{C_{1}}\big)^{\frac{\beta}{\alpha}}h^{1-\beta}(0)
+h(0)\Big(1-(\frac{C_{2}}{C_{1}}\big)^{\frac{\beta}{\alpha}}h^{-\beta}(0)\Big)e^{-C_{2}\big(\frac{C_{1}}{C_{2}}\big)^{\frac{\beta}{\alpha}}h^{\beta-1}(0)t},&0<\beta\leq1.
\end{cases}
\end{align}
\end{lemma}
\begin{proof}

{\bf Case 1.} If $h(0)\leq
\Big(\frac{C_{2}}{C_{1}}\Big)^{\frac{1}{\alpha}},$ then we have the following
claim
\begin{align}\label{lem501-3} h(t)\leq
\Big(\frac{C_{2}}{C_{1}}\Big)^{\frac{1}{\alpha}},~\mathrm{for}~t\geq0. \end{align}

If not, then there would exist a $s_{0}>0$ and a $s_{1}>0$ such that
\begin{align}\label{lem501-4} h(s_{0})=
\Big(\frac{C_{2}}{C_{1}}\Big)^{\frac{1}{\alpha}},~h(t)>
\Big(\frac{C_{2}}{C_{1}}\Big)^{\frac{1}{\alpha}},~\mathrm{for}~s_{0}<t<s_{1}. \end{align}
Thus there exists a $s_{2}\in(s_{0},s_{1})$ such that
\begin{align}\label{lem501-5}
h'(s_{2})=\frac{h(s_{1})-h(s_{0})}{s_{1}-s_{0}}>0.\end{align}
On the other hand, noting that $h(t)>
\Big(\frac{C_{2}}{C_{1}}\Big)^{\frac{1}{\alpha}}\geq\Big(\frac{C_{2}}{C_{1}}\Big)^{\frac{1}{\beta}}~\mathrm{for}~s_{0}<t<s_{1}$, we have
$$C_{1} \min\{h^{\alpha}(t),h^{\beta}(t)\}>C_{2}\Longrightarrow h'(t)<0,~\mathrm{for}~s_{0}<t<s_{1}.$$  This is a
contradiction.

{\bf Case 2.} If $h(0)>\Big(\frac{C_{2}}{C_{1}}\Big)^{\frac{1}{\alpha}},$
then we claim that either
\begin{align}\label{lem501-6} h(t)\geq
\Big(\frac{C_{2}}{C_{1}}\Big)^{\frac{1}{\alpha}},~\mathrm{for}~t\geq0,\end{align}
 or there
exists a $s^{*}>0$ such that
\begin{align}\label{lem501-7}
h(t)>\Big(\frac{C_{2}}{C_{1}}\Big)^{\frac{1}{\alpha}},~\mathrm{for}~0<t<s^{*};~
h(t)\leq\Big(\frac{C_{2}}{C_{1}}\Big)^{\frac{1}{\alpha}},~\mathrm{for}~t\geq
s^{*}.
\end{align}
If \eqref{lem501-6} holds, we have $ h(t)\geq
\Big(\frac{C_{2}}{C_{1}}\Big)^{\frac{1}{\alpha}}\geq\Big(\frac{C_{2}}{C_{1}}\Big)^{\frac{1}{\beta}},~\mathrm{for}~t\geq0.$
So
\begin{align*}\label{lem501-8}
h'(t)\leq C_{2}-C_{1} \min\{h^{\alpha}(t),h^{\beta}(t)\}<0\Longrightarrow
h(t)\leq  h(0).
\end{align*}
If \eqref{lem501-7} holds, the proof of this part is analogous to that in Case
1 and \eqref{lem501-6}. In summary, we have $$h(t)\leq  h(0).$$

In the forthcoming proof, we analyze the asymptotic behavior of the solution under the assumption that $h(0)>\Big(\frac{C_{2}}{C_{1}}\Big)^{\frac{1}{\alpha}}$. At first, we consider the case $\beta>1$. In such case,  we claim the following inequality holds:
\begin{equation}\label{lem501-9}
h(t)\leq\Big(\frac{C_{2}}{C_{1}}\Big)^{\frac{1}{\alpha}}+\Big[\Big(h(0)-\Big(\frac{C_{2}}{C_{1}}\Big)^{\frac{1}{\alpha}}\Big)^{1-\beta}
+C_{1}\Big(\frac{C_{2}}{C_{1}}\Big)^{\frac{\alpha-\beta}{\alpha-\beta+1}}(\beta-1)t\Big]^{\frac{1}{1-\beta}}.
\end{equation}
If \eqref{lem501-6} holds, we have $C_{1}h^{\alpha}(t)\geq C_{2}$. We let $z(t)=h(t)-\Big(\frac{C_{2}}{C_{1}}\Big)^{\frac{1}{\alpha-\beta+1}}\geq0, \mathrm{for}~t\geq 0,$ and then we apply the inequality $(a+b)^{\beta}\geq a^{\beta}+b^{\beta},~a,b\geq0,~\beta>1$ to obtain the following inequality
\begin{align*}
z'(t)+C_{1}\Big(\frac{C_{2}}{C_{1}}\Big)^{\frac{\alpha-\beta}{\alpha-\beta+1}}z^{\beta}(t)+C_{2}&\leq z'(t)+C_{1}\Big(\frac{C_{2}}{C_{1}}\Big)^{\frac{\alpha-\beta}{\alpha-\beta+1}}\Bigg(z(t)+\Big(\frac{C_{2}}{C_{1}}\Big)^{\frac{1}{\alpha-\beta+1}}\Bigg)^{\beta}\\
&=h'(t)
+C_{1}\Big(\frac{C_{2}}{C_{1}}\Big)^{\frac{\alpha-\beta}{\alpha-\beta+1}} h^{\beta}(t)\\
&\leq h'(t)+C_{1} \min\{h^{\alpha}(t),h^{\beta}(t)\}\leq C_{2},
\end{align*}
which implies that
\begin{align}\label{lem501-10}
z'(t)+C_{1}\Big(\frac{C_{2}}{C_{1}}\Big)^{\frac{\alpha-\beta}{\alpha-\beta+1}}z^{\beta}(t)\leq 0.
\end{align}
Combining \eqref{lem501-10} with $\hbox{Gronwall's}$ inequality, we have
\begin{align}\label{lem501-11}
z(t)\leq  \Big(z^{1-\beta}(0)+C_{1}\Big(\frac{C_{2}}{C_{1}}\Big)^{\frac{\alpha-\beta}{\alpha-\beta+1}}(\beta-1)t\Big)^{\frac{1}{1-\beta}}.
\end{align}
The first inequality in \eqref{lem501-2} follows from \eqref{lem501-11}.

If \eqref{lem501-7} holds, for $0<t<s^{*}$, we follow the lines of the proof of \eqref{508} to get
\begin{align}\label{lem501-12}
z(t)\leq  \Big(z^{1-\beta}(0)+C_{1}\Big(\frac{C_{2}}{C_{1}}\Big)^{\frac{\alpha-\beta}{\alpha-\beta+1}}(\beta-1)t\Big)^{\frac{1}{1-\beta}},~\mathrm{for}~0<t\leq  s^{*}.
\end{align}
For $t\geq  s^{*}$, it is obvious that
\begin{align}\label{lem501-13}
h(t)\leq\Big(\frac{C_{2}}{C_{1}}\Big)^{\frac{1}{\alpha}}+\Big[\Big(h(0)-\Big(\frac{C_{2}}{C_{1}}\Big)^{\frac{1}{\alpha}}\Big)^{1-\beta}
+C_{1}\Big(\frac{C_{2}}{C_{1}}\Big)^{\frac{\alpha-\beta}{\alpha-\beta+1}}(\beta-1)t\Big]^{\frac{1}{1-\beta}}.
\end{align}
Immediately, the first inequality \eqref{lem501-2} is also obtained from \eqref{lem501-12} and \eqref{lem501-13}.

For $0<\beta\leq 1$, the proof of this part is analogous to that of \eqref{513}, we omit it here. The proof of Lemma \ref{lem501} is complete.
\end{proof}

For simplicity, let us recall that $B_{0}$ is the embedding constant of  $\|u\|_{2}\leq B_{0}\|\nabla u\|_{p(.)},~$for $u\in V(\Omega)$ and define
\begin{equation*}
\begin{split}
M_{1}=\max\Big\{(2B^{r^{+}})^{\frac{p^{-}}{p^{-}-r^{+}}},(2B^{r^{+}})^{\frac{p^{+}}{p^{+}-r^{+}}},(2B^{r^{-}})^{\frac{p^{-}}{p^{-}-r^{-}}},(2B^{r^{-}})^{\frac{p^{+}}{p^{+}-r^{-}}}\Big\},
\end{split}
\end{equation*}
we have the following results.
\begin{theorem}\label{Thm500}
Suppose that the exponents $p(x),r(x)$ satisfy
$\eqref{addequ101}$ and $1<r^{+}<p^{-}$, then the solution
of Problem $\eqref{equ101}$ exists globally for arbitrary initial energy. Further, the following asymptotic estimates hold for any $t>0$

(i) if $M_{1}\geq\frac{1}{2}$, then
\begin{equation*}
\frac{\|u(t)\|^{2}_{2}}{B_{0}^{2}}\leq\begin{cases}
(2M_{1})^{\frac{2}{p^{-}}},&p^{-}>1,~\frac{\|u_{0}\|_{2}^{2}}{B_{0}^{2}}\leq (2M_{1})^{\frac{2}{p^{-}}};\\
(2M_{1})^{\frac{2}{p^{-}}}+\Big[\Big(\frac{\|u_{0}\|_{2}^{2}}{B_{0}^{2}}-(2M_{1})^{\frac{2}{p^{-}}}\Big)^{\frac{2-p^{-}}{2}}
+\frac{(p^{-}-2)t}{2B_{0}^{2}}\Big]^{\frac{2}{2-p^{-}}},&p^{-}>2,~\frac{\|u_{0}\|_{2}^{2}}{B_{0}^{2}}> (2M_{1})^{\frac{2}{p^{-}}};\\
2M_{1}(\frac{\|u_{0}\|_{2}^{2}}{B_{0}^{2}})^{\frac{2-p^{-}}{2}}+\frac{\|u_{0}\|_{2}^{2}}{B_{0}^{2}}
\Big(1-\frac{2M_{1}B_{0}^{p^{-}}}{\|u_{0}\|_{2}^{p^{-}}}\Big)e^{-\frac{\|u_{0}\|_{2}^{p^{-}-2}}{B^{p^{-}}_{0}}t},&p^{-}\leq2,~
\frac{\|u_{0}\|_{2}^{2}}{B_{0}^{2}}> (2M_{1})^{\frac{2}{p^{-}}};
\end{cases}
\end{equation*}

(ii) if $M_{1}<\frac{1}{2}$, then
\begin{equation*}
\frac{\|u(t)\|^{2}_{2}}{B_{0}^{2}}\leq\begin{cases}
(2M_{1})^{\frac{2}{p^{+}}},&p^{-}>1,~\frac{\|u_{0}\|_{2}^{2}}{B_{0}^{2}}\leq (2M_{1})^{\frac{2}{p^{+}}};\\
(2M_{1})^{\frac{2}{p^{+}}}+\Big[\Big(\frac{\|u_{0}\|_{2}^{2}}{B_{0}^{2}}-(2M_{1})^{\frac{2}{p^{+}}}\Big)^{\frac{2-p^{-}}{2}}
+C_{3}t\Big]^{\frac{2}{2-p^{-}}},&p^{-}>2,~\frac{\|u_{0}\|_{2}^{2}}{B_{0}^{2}}> (2M_{1})^{\frac{2}{p^{+}}};\\
(2M_{1})^{\frac{p^{-}}{p^{+}}}\frac{\|u_{0}\|_{2}^{2-p^{-}}}{B_{0}^{2-p^{-}}}
+\frac{\|u_{0}\|_{2}^{2}}{B_{0}^{2}}\Big(1-\frac{(2M_{1})^{\frac{p^{-}}{p^{+}}}B_{0}^{p^{-}}}{\|u_{0}\|_{2}^{p^{-}}}\Big)
e^{-C_{4}t},&p^{-}\leq2,~
\frac{\|u_{0}\|_{2}^{2}}{B_{0}^{2}}> (2M_{1})^{\frac{2}{p^{+}}},
\end{cases}
\end{equation*}
where
\begin{align*}
C_{3}=(2M_{1})^{\frac{p^{+}-p^{-}}{p^{+}-p^{-}+2}}\frac{(p^{-}-2)}{2B_{0}^{2}},
C_{4}=\frac{\|u_{0}\|^{p^{-}-2}_{2}}{B^{p^{-}}_{0}}(2M_{1})^{\frac{p^{+}-p^{-}}{p^{+}}}.
\end{align*}
\end{theorem}
\begin{proof} The proof of this theorem will be divided into some steps.

{\bf Step 1. Establish a differential inequality.}
Let
 $G(t)=\int_{\Omega}|u|^{2}dx.$  According to Definition \ref{def301}, we have
\begin{equation}\label{Thm5001}
\begin{split}
G'(t)&=2\int_{\Omega}u u_{t}dx =-2\int_{\Omega}|\nabla
u|^{p(x)}dx+2\int_{\Omega}|u|^{r(x)}dx.
\end{split}
\end{equation}
On the one hand, let us recall \eqref{equ306}, then
\begin{equation}\label{Thm5002}
\begin{split}
\int_{\Omega}|u|^{r(x)}dx &\leq  \max\Bigg\{B^{r^{+}}\max\Big\{\big(\int_{\Omega}|\nabla u|^{p(x)}dx\big)^{\frac{r^{+}}{p^+}},\big(\int_{\Omega}|\nabla u|^{p(x)}dx\big)^{\frac{r^{+}}{p^-}}\Big\},\\
&~~~~B^{r^{-}}\max\Big\{\big(\int_{\Omega}|\nabla u|^{p(x)}dx\big)^{\frac{r^{-}}{p^+}},\big(\int_{\Omega}|\nabla u|^{p(x)}dx\big)^{\frac{r^{-}}{p^-}}\Big\}  \Bigg\}.
\end{split}
\end{equation}
On the other hand, it follows from  $r^{+}<p^{-}$ and  Young's inequality
\begin{align}\label{Thm5003}
\nonumber&\max\Bigg\{B^{r^{+}}\max\Big\{\big(\int_{\Omega}|\nabla u|^{p(x)}dx\big)^{\frac{r^{+}}{p^+}},\big(\int_{\Omega}|\nabla u|^{p(x)}dx\big)^{\frac{r^{+}}{p^-}}\Big\},\\
&~~~~B^{r^{-}}\max\Big\{\big(\int_{\Omega}|\nabla u|^{p(x)}dx\big)^{\frac{r^{-}}{p^+}},\big(\int_{\Omega}|\nabla u|^{p(x)}dx\big)^{\frac{r^{-}}{p^-}}\Big\}  \Bigg\}
\nonumber\\&\leq\frac{1}{2}\int_{\Omega}|\nabla
u|^{p(x)}dx+\max\Big\{(2B^{r^{+}})^{\frac{p^{-}}{p^{-}-r^{+}}},(2B^{r^{+}})^{\frac{p^{+}}{p^{+}-r^{+}}},(2B^{r^{-}})^{\frac{p^{-}}{p^{-}-r^{-}}},(2B^{r^{-}})^{\frac{p^{+}}{p^{+}-r^{-}}}\Big\}.
\end{align}

Therefore, \eqref{Thm5001}-\eqref{Thm5003} indicate that
\begin{equation}\label{Thm5004}
G'(t)+\int_{\Omega}|\nabla
u|^{p(x)}dx\leq2M_1.
\end{equation}
Once again, we apply \eqref{PRO102} and Sobolev embedding inequality $\|u\|_{2}\leq B_{0}\|\nabla u\|_{p(.)}$ to obtain
\begin{align}\label{Thm5005}
\min\Bigg\{\Big(\frac{\|u\|^{2}_{2}}{B^{2}_{0}}\Big)^{\frac{p^{-}}{2}},\Big(\frac{\|u\|^{2}_{2}}{B^{2}_{0}}\Big)^{\frac{p^{+}}{2}}\Bigg\}\leq \int_{\Omega}|\nabla
u|^{p(x)}dx.
\end{align}
Set $w(t)=\frac{\|u\|^{2}_{2}}{B^{2}_{0}}$. Then, \eqref{Thm5004} and \eqref{Thm5005} show that
\begin{align}\label{Thm5006}
w'(t)+\frac{1}{B^{2}_{0}}\min\Big\{w^{\frac{p^{-}}{2}},w^{\frac{p^{+}}{2}}\Big\}\leq\frac{2M_{1}}{B^{2}_{0}}.
\end{align}

{\bf Step 2. Asymptotic estimate.} In the forthcoming proof the cases that $M_{1}\geq\frac{1}{2}$ and $M_{1}<\frac{1}{2}$ will be discussed
separately.

(i) If $M_{1}\geq\frac{1}{2}$, then, by means of \eqref{Thm5006} and Lemma \ref{lem500}, we have
\begin{equation*}
w(t)\leq\begin{cases}
(2M_{1})^{\frac{2}{p^{-}}}+\Big[(w(0)-(2M_{1})^{\frac{2}{p^{-}}})_{+}^{\frac{2-p^{-}}{2}}
+\frac{(p^{-}-2)t}{2B_{0}^{2}}\Big]^{\frac{2}{2-p^{-}}},&~~p^{-}>2;\\
2M_{1}w^{\frac{2-p^{-}}{2}}(0)+w(0)\Big(1-2M_{1}
w^{\frac{-p^{-}}{2}}(0)\Big)_{+}e^{-\frac{w^{\frac{p^{-}-2}{2}}(0)}{B^{2}_{0}}t},&~~1<p^{-}\leq2,
\end{cases}
\end{equation*}
where $Z_{+}=\max\{Z,0\}.$

(ii) If $M_{1}<\frac{1}{2}$, then, by means of \eqref{Thm5006} and Lemma \ref{lem501}, we have
\begin{equation*}
w(t)\leq\begin{cases}
(2M_{1})^{\frac{2}{p^{+}}}+\Big[(w(0)-(2M_{1})^{\frac{2}{p^{+}}})_{+}^{\frac{2-p^{-}}{2}}
+C_{3}t\Big]^{\frac{2}{2-p^{-}}},&~~p^{-}>2;\\
(2M_{1})^{\frac{p^{-}}{p^{+}}}w^{\frac{2-p^{-}}{2}}(0)+w(0)
\Big(1-(2M_{1})^{\frac{p^{-}}{p^{+}}}w^{\frac{-p^{-}}{2}}(0)\Big)_{+}
e^{-C_{4}t},&~~1<p^{-}\leq2.
\end{cases}
\end{equation*}
\end{proof}
At last, we discuss the case $r^{-}\leq \min\{p^{+},2\}.$  Our main result is as follows:
\begin{theorem}\label{Thm502}
Suppose that $p(x),r(x)$ satisfy
$\eqref{addequ101}$ and $1<r^{-}\leq \min\{p^{+},2\},~r^{+}<2$, then the solution
of Problem $\eqref{equ101}$ exists globally for negative initial energy. Further, the following asymptotic estimates hold for any $t>0$
\begin{align*}
\Big(\|u_{0}\|_{2}^{2}+\frac{2p^{+}E(0)}{C_{5}}\Big)e^{-C_{5}t}-\frac{2p^{+}E(0)}{C_{5}}\leq\|u\|_{2}^{2}
\leq\left[\frac{C_{7}}{C_{6}}+\Big(\|u_{0}\|_{2}^{2-r^{+}}-\frac{C_{7}}{C_{6}}\Big)e^{-C_{6}t}\right]^{\frac{2}{2-r^{+}}},
\end{align*}
where the coefficients $C_{5}, C_{6}$ and $C_{7}$ will be defined in \eqref{Thm5028}  and \eqref{Thm5031}.
\end{theorem}
\begin{proof} This proof will be divided into some steps.

{\bf Step 1. Establish a differential inequality about $\|u\|_{2}.$ }

By means of Lemma \ref{PRO102}, we get
\begin{equation}\label{Thm5022}
\int_{\Omega}|u|^{r(x)}dx\leq \max\{\|u\|^{r^{-}}_{r(.)},\|u\|^{r^{+}}_{r(.)}\}
\leq (1+|\Omega|)^{r^{+}}\max\{\|u\|_{2}^{r^{-}},\|u\|_{2}^{r^{+}}\};
\end{equation}
\begin{equation}\label{addThm5022}
\|u\|_{2}\leq  B_{0}\|\nabla u\|_{p(.)}\leq
B_{0}\max\Big\{(\int_{\Omega}|\nabla
u|^{p(x)}dx)^{\frac{1}{p^{+}}},~(\int_{\Omega}|\nabla
u|^{p(x)}dx)^{\frac{1}{p^{-}}}\Big\}.
\end{equation}

Thus, \eqref{Thm5001} and \eqref{Thm5022} \eqref{addThm5022} yield
\begin{align}\label{Thm5023}
G'(t)+2\min\Bigg\{\Big(\frac{G(t)}{B_{0}^{2}}\Big)^{\frac{p^{+}}{2}},\Big(\frac{G(t)}{B_{0}^{2}}\Big)^{\frac{p^{-}}{2}}\Bigg\}
\leq
2(1+|\Omega|)^{r^{+}}\max\Bigg\{G^{\frac{r^{+}}{2}}(t),G^{\frac{r^{-}}{2}}(t)\Bigg\}.
\end{align}

{\bf Step 2. Build up a lower estimate for $G(t)$.}
Let us first recall \eqref{Thm5001}, we get
\begin{equation}\label{Thm5024}
\begin{split}
G'(t)&=-2\int_{\Omega}|\nabla
u|^{p(x)}dx+2\int_{\Omega}|u|^{r(x)}dx\\
&\geq\frac{2(r^{-}-p^{+})}{r^{-}}\int_{\Omega}|u|^{r(x)}dx-2p^{+}E(0).
\end{split}
\end{equation}

Noting that $r(x)\leq 2$, \eqref{Thm5022} yields
\begin{align*}
\int_{\Omega}|u|^{r(x)}dx\leq(1+|\Omega|)^{r^{+}}\max\{G^{\frac{r^{+}}{2}}(t),G^{\frac{r^{-}}{2}}(t)\}.
\end{align*}
In \eqref{Thm5024}, we replace $\int_{\Omega}|u|^{r(x)}dx$  with $G(t)$ to obtain
\begin{align}\label{Thm5025}
G'(t)\geq\frac{2(r^{-}-p^{+})}{r^{-}}(1+|\Omega|)^{r^{+}}\max\{G^{\frac{r^{+}}{2}}(t),G^{\frac{r^{-}}{2}}(t)\}-2p^{+}E(0).
\end{align}
Obviously,  it follows from Lemma 4.4 of \cite{GG} and \eqref{Thm5025}
\begin{align}\label{Thm5026}
G(t)\geq\Bigg\{G(0),~\Big(\frac{p^{+}r^{-}E(0)}{(r^{-}-p^{+})(1+|\Omega|)^{r^{+}}}\Big)^{\frac{2}{r^{+}}},
\Big(\frac{p^{+}r^{-}E(0)}{(r^{-}-p^{+})(1+|\Omega|)^{r^{+}}}\Big)^{\frac{2}{r^{-}}}\Bigg\}:=M_{2}>0.
\end{align}

{\bf Step 3. Lower estimate of asymptotic behavior.} In fact, by \eqref{Thm5024}-\eqref{Thm5026}, we have
\begin{align*}
G'(t)+\frac{2(p^{+}-r^{-})}{r^{-}}(1+|\Omega|)^{r^{+}}\max\{M_{2}^{\frac{r^{+}-2}{2}},M_{2}^{\frac{r^{-}-2}{2}}\}G(t)\geq-2p^{+}E(0),
\end{align*}
which  implies that
\begin{align}\label{Thm5027}
G(t)\geq\Big(G(0)+\frac{2p^{+}E(0)}{C_{5}}\Big)e^{-C_{5}t}-\frac{2p^{+}E(0)}{C_{5}},
\end{align}
where
\begin{align}\label{Thm5028}
C_{5}=\frac{2(p^{+}-r^{-})}{r^{-}}(1+|\Omega|)^{r^{+}}\max\{M_{2}^{\frac{r^{+}-2}{2}},M_{2}^{\frac{r^{-}-2}{2}}\}.
\end{align}

{\bf Step 4. Upper estimate of asymptotic behavior.} Utilizing \eqref{Thm5023} and \eqref{Thm5026} and making a simple computation, we get
\begin{align}\label{Thm5029}
G'(t)&\nonumber+\frac{2}{B_0^{2}}
\min\Bigg\{\Big(\frac{M_{2}}{B_0^{2}}\Big)^{\frac{p^{+}-2}{2}},
\Big(\frac{M_{2}}{B_0^{2}}\Big)^{\frac{p^{-}-2}{2}}\Bigg\}G(t)\\
&\leq2B_0^{-r^{+}}(1+|\Omega|)^{\frac{r^{-}}{2}}
\max\Bigg\{\Big(\frac{M_{2}}{B_0^{2}}\Big)^{\frac{r^{-}-r^{+}}{2}},1\Bigg\}G^{\frac{r^{+}}{2}}(t),
\end{align}
which is equivalent to
\begin{align}\label{Thm5030}
L'(t)+C_{6}L(t)\leq C_{7},~~L(t)=G^{1-\frac{r^{+}}{2}}(t)
\end{align}
where
\begin{align}\label{Thm5031}
C_{6}&\nonumber=\frac{2-r^{+}}{B_0^{2}}\min\Bigg\{\Big(\frac{M_{2}}{B_0^{2}}\Big)^{\frac{p^{+}-2}{2}},
\Big(\frac{M_{2}}{B_0^{2}}\Big)^{\frac{p^{-}-2}{2}}\Bigg\},\\
C_{7}&=\max\Bigg\{\Big(\frac{M_{2}}{B_0^{2}}\Big)^{\frac{r^{-}-r^{+}}{2}},1\Bigg\}(2-r^{+})B_0^{-r^{+}}(1+|\Omega|)^{\frac{r^{+}}{2}}.
\end{align}
It is not hard to check that \eqref{Thm5030} may rewritten as
\begin{align}\label{Thm5032}
G(t)\leq\left[\frac{C_{7}}{C_{6}}+\Big(G^{\frac{2-r^{+}}{2}}(0)-\frac{C_{7}}{C_{6}}\Big)e^{-C_{6}t}\right]^{\frac{2}{2-r^{+}}},~t\geq0.
\end{align}
This proof is complete.
\end{proof}
\section{Comments}
We have established some new results on the global and non-global existence, extending \cite{GG}. Furthermore, we have also obtained some estimates of asymptotic behavior. Those results are new even when the exponents $p,~r$ are fixed constants.  In addition, our method proposed are directly applied to study some problems of \cite{RFAPMPJD,JXYCHJ,FCLCHX,QBZ} and the following problem
\begin{equation}\label{6304}
\begin{cases}
u_{t}=\mathrm{div}(|\nabla u|^{p(x)-2}\nabla u)+|u|^{r(x)-2}u,&(x,t)\in Q_{T},\\
u(x,t)=0,&(x,t)\in\Gamma_{T},\\
u(x,0)=u_{0}(x),&x\in\Omega.
\end{cases}
\end{equation}
We can replace $V(\Omega)$ by $W^{1,p(x)}_{0}(\Omega)=:\Big\{u\in W^{1,p(x)}(\Omega):~u=0,x\in\partial\Omega\Big\}.$  We define the associated functionals as follows:
\begin{equation*}
J(u)=\int_{\Omega}\frac{1}{p(x)}|\nabla u|^{p(x)}dx-\int_{\Omega}\frac{1}{r(x)}|u|^{r(x)}dx;
\end{equation*}
\begin{equation*}
\begin{split}
I(u)=\int_{\Omega}|\nabla u|^{p(x)}dx-\int_{\Omega}|u|^{r(x)}dx.
\end{split}
\end{equation*}
The rest of the discussion is the same as that of this paper and may be left to the readers.

$\bullet$ As a matter of fact,  we find out that
the anisotropy of the exponent $p(x)$ may bring many new problems.
{\bf For example, whether or not is the following identity true?}
 \begin{align}\label{6305}
 d=\inf\limits_{0\not\equiv u\in V(\Omega)}\sup\limits_{\lambda>0}J(\lambda u).
 \end{align}
Actually, it is well known that the identity above is always true for constant exponents $p,~r.$ However,  for non-constant exponent $p(x)$,  it is not clear whether the identity above is true or not. The main reason is that the modular version of Poincar\'{e} inequality is not valid.  In other words, the following inequality
\begin{align*}
\int_{\Omega}|u|^{p(x)}dx\leqslant C\int_{\Omega}|\nabla
u|^{p(x)}dx,~~u\in V(\Omega),
\end{align*}
is not always true. For example, set $B_{k}=\Big\{X=(x,y,z)\in R^{3}:|X|^{2}=x^{2}+y^{2}+z^{2}\leqslant k^{2}\Big\},k=1,2,3,$ and then
we may construct that two Lipschitz functions $p(X),~u(X)$ satisfy
\begin{equation*}
p(X)=\begin{cases}
\frac{3}{2}+|X|,~~&X\in B_{1},\\
\frac{5}{2},~~&X\in B_{2}\setminus B_{1},\\
\frac{9}{2}-|X|,~~&X\in B_{3}\setminus B_{2},
\end{cases}
~~\mathrm{and}~~ u(X)=\begin{cases}
\frac{3}{4}-|X|,~~&X\in B_{1},\\
-\frac{1}{4},~~&X\in B_{2}\setminus B_{1},\\
\frac{1}{172}(104|X|-251),~~&X\in B_{3}\setminus B_{2},
\end{cases}
\end{equation*}
where $B_{i}\setminus B_{j}:=\Big\{X\in R^{3}~: X\in B_{i}~\textrm{and}~X\not\in B_{j},~~j<i\Big\}.$

For $\varepsilon>0$, set $v_{\varepsilon}(X)=\varepsilon u(X)$. Then, a direct computation shows that
\begin{align*}
\iint_{B_{3}}|\nabla v_{\varepsilon}(X)|^{p(X)}dX&=
\iint_{B_{1}}|\nabla v_{\varepsilon}(X)|^{p(X)}dX
+\iint_{B_{3}\setminus B_{2}}|\nabla v_{\varepsilon}(X)|^{p(X)}dX\\
&\leqslant\frac{4\pi\varepsilon^{\frac{3}{2}}(\varepsilon-1)}{\ln \varepsilon}+
\frac{36\pi\varepsilon^{\frac{3}{2}}(\varepsilon-1)}{\ln \varepsilon}=
\frac{40\pi\varepsilon^{\frac{3}{2}}(\varepsilon-1)}{\ln \varepsilon};\\
\iint_{B_{3}}|v_{\varepsilon}(X)|^{p(X)}dX&\geqslant\iint_{B_{2}\setminus B_{1}}|v_{\varepsilon}(X)|^{p(X)}dX=
\frac{7\pi}{24}\varepsilon^{\frac{5}{2}}.
\end{align*}
Then, the quotient
 \begin{align*}
 \frac{\iint_{B_{3}}|\nabla v_{\varepsilon}(X)|^{p(X)}dX}{\iint_{B_{3}}|v_{\varepsilon}(X)|^{p(X)}dX}
 \leqslant\frac{\frac{40\pi\varepsilon^{\frac{3}{2}}(\varepsilon-1)}{\ln \varepsilon}}{\frac{7\pi}{24}\varepsilon^{\frac{5}{2}}}
 =\frac{960(\varepsilon-1)}{7\varepsilon\ln\varepsilon}\rightarrow0,~~\varepsilon\rightarrow\infty,
 \end{align*}
 which implies that
 \begin{align*}
 \inf\limits_{0\not\equiv u\in V(B_{3})}\frac{\iint_{B_{3}}|\nabla u(X)|^{p(X)}dX}{\iint_{B_{3}}|u(X)|^{p(X)}dX}=0.
  \end{align*}
  For $u\in W^{1,p(x)}_{0}(\Omega),$ similar examples may be referred to
  \cite[Example 8.2.7]{LDPHPM} or \cite[p.444-p.445]{XLF}.

$\bullet$ In addition,  we need to point out that the method used in this paper can not directly be applied to the case that the exponents depend on the space variable and time variable because the monotonicity of the energy functional fails(in other words, one can not obtain Lemma \ref{lem301}).

$\bullet$ Finally, when we finished this paper, the anonymous referee told us that the authors of \cite{LCNQVC} applied the similar method to analyze the blow-up and global existence of solutions to Problem \eqref{6304}. In \cite{LCNQVC}, the authors claimed that \eqref{6305} holds.  {\bf However, the validity of such claim is dubious.} In fact, as far as we know, the key step to prove \eqref{6305} is that one needs to compute the maximum point  $\lambda^{*}=\lambda^{*}(u)>0$ such that
$J(\lambda^{*}u)=\sup\limits_{\lambda>0}J(\lambda u).$  Therefore, for constant-exponents cases, it is not difficulty to verify that $\lambda^{*}$ satisfies
\begin{align*}
\int_{\Omega}|\nabla u|^{p}dx=(\lambda^{*})^{r-p}\int_{\Omega}|u|^{r}dx.
\end{align*}
However, the lack of homogeneity of the modular(that is $\int_{\Omega}|\lambda u|^{p(x)}dx\not\equiv\lambda^{p(x)}\int_{\Omega}|u|^{p(x)}dx$) makes the maximum point $\lambda^{*}$  depend ont only on the quotient $\frac{\|\nabla u\|_{p}}{\|u\|_{r}}$, but also on the space variable $x$ in the variable exponent case. So we think that the proofs of (ii) of  Lemma 4.1, Lemma 4.2  and Theorem 5.5 in \cite{LCNQVC} need to be carefully re-examined.

{\bf Acknowledgments.} The authors express their thanks to
the referees for their helpful comments and suggestions and to Prof. Wenjie Gao and Prof. Yuzhu Han for heplful discussions.

\end{document}